\documentclass[12pt]{article}
 \usepackage{amsmath}
 \usepackage{blkarray}
  \usepackage{amsfonts}
\usepackage{tikz}
\usepackage[title]{appendix}
\usepackage{subfig}
\usepackage{array}
\usepackage{tabu}
\usepackage{color} 
\usepackage[
  separate-uncertainty = true,
  multi-part-units = repeat
]{siunitx}

\usetikzlibrary{arrows}
\usepackage{epsfig,}
\usepackage{epsfig}
\tikzset{
    vertex/.style = {
        circle,
        draw,
        outer sep = 3pt,
        inner sep = 3pt,
    },edge/.style = {->,> = latex'}
}
\usepackage{amssymb}
\usepackage{enumitem}
\usepackage{amsthm}
\usepackage{dsfont}
\newcommand{\rr}{\mathbb{R}}
\newcommand{\1}{\mathbf{1}}
\newcommand{\0}{\mathbf{0}}

\newcommand{\El}{\mathrm{El}}

\def\det{{\rm det}}

\newtheorem{theorem}{Theorem}

\newtheorem{corollary}{Corollary}

\newtheorem{example}{Example}

\newtheorem{lemma}{Lemma}
\newtheorem{definition}{Definition}

\begin{document}
\begin{center}
\begin{large}
 Distance matrix of enhanced power graphs of finite groups
\end{large}
\end{center}
\begin{center}
Anita Arora, Hiranya Kishore Dey, Shivani Goel\\
\today
\end{center}

\begin{abstract}
The enhanced power graph of a group $G$ is the graph 
$\mathcal{G}_E(G)$  with vertex set $G$ and edge set $ \{(u,v): u, v \in \langle w \rangle,~\mbox{for some}~ w \in G\}$. In this paper, we compute the spectrum of the distance matrix of the enhanced power graph of non-abelian groups of order $pq$, dihedral groups, dicyclic groups, elementary abelian groups $\El(p^n)$ and the non-cyclic abelian groups $\El(p^n)\times \El(q^m)$ and $\El(p^n)\times \mathbb{Z}_m$, where $p$ and $q$ are distinct primes. 

For the non-cyclic abelian group $\El(p^n)\times \El(q^m)$, we also compute the spectrum of the adjacency matrix of its enhanced power graph and the spectrum of the adjacency and the distance matrix of its power graph.
\end{abstract} 


{\bf AMS CLASSIFICATION.} 05C50, 05C25, 20K01


\section{Introduction}

The study of graphs associated with various algebraic structures has been a topic of increasing interest during the last couple of decades. Cayley graphs are one of the well-known classes of graphs arising from algebraic structures. Many different types of graphs, such as power graph \cite{directedgrphcompropofsemgrpkq3,combinatorialpropertyandpowergraphsofgroupskq1} and enhanced power graph \cite{AalipourcameronEJC} of a group has been introduced to explore the properties of algebraic structures using graph theory. It is important to notice that these graphs represent some combinatorial properties of the underlying group. Research on combinatorial properties of groups and semigroups originates from Bernhard Neumann’s theorem \cite{Neumann1976APO} answering a question of Paul Erd\'os. In \cite{David}, a brief theory of applications of combinatorial properties of groups is given. We are interested in studying power graphs and enhanced power graphs as they represent some combinatorial properties of groups. Also, these graphs help us to
characterize the algebraic structures with isomorphic graphs and realize 
the interdependence between the algebraic structures and the corresponding graphs. 
The concept of power graph was introduced in the context of semigroup theory by Kelarev and Quin \cite{combinatorialpropertyandpowergraphsofgroupskq1}.

\begin{definition}\label{defn: powr graph}
The \emph{power graph} $\mathcal{P}(G)$ of a group $G$ is a simple graph with vertex set $G$ and two vertices $a$ and $b$ are adjacent if and only if one of them is some power of the other, i.e., either $b^k = a$ or $a^k = b$, for some $k \in \mathbb{N}$. 
\end{definition} 

Another closely related graph in this context is the enhanced power graph of a group which is introduced by Alipour et al.  
\cite{AalipourcameronEJC}.
\begin{definition}\label{defn: enhcdpowr graph}
The \emph{enhanced power graph} $\mathcal{G}_E(G)$ of a group $G$ is the graph with vertex set $G,$ and two vertices $a$ and $b$ are adjacent if and only if
 $a , b \in  \langle c \rangle$, for some $c \in G$.
\end{definition}

Various properties of the enhanced power graph of finite groups have been studied in detail. Aalipour et al. \cite{AalipourcameronEJC} 
characterized  finite groups for which the power and the enhanced power graphs are equal.
Besides, Zahirovic et al. \cite{Zahirovienhnacedpwrgraph} proved that finite groups with isomorphic enhanced power graphs have isomorphic directed power graphs.
Bera et al. \cite{enhancedpwrgrapbb3} studied the completeness, dominatability, and many other interesting properties of the enhanced power graph. 
The metric dimension 
of enhanced power graphs of finite groups was studied by Ma and She \cite{ma-She}. 

In the last decade, many researchers have studied matrices related to the power and enhanced power graph of finite groups. Chattopadhyay and Panigrahi \cite{chatto-pani-LAMA2015} studied the Laplacian spectra of power
graphs of finite cyclic groups and the dihedral groups. Mehranian et al. \cite{mehranian-gholami-ashrafi-LAMA2016} computed the  adjacency spectrum of the power graph of cyclic groups, dihedral groups, elementary abelian groups of prime power order and the Mathieu group.
Hamzeh and Ashrafi \cite{hamzehashrafiFILOMAT2017} investigated adjacency and Laplacian spectra of power
graphs of the cyclic, dihedral and quaternion groups. 

Wani and Shrivastav \cite{wani2020} studied the spectrum of the distance matrix of power graphs of cyclic groups, dihedral groups of order $2p^n$, dicyclic groups of order $4p^n$, and groups of order $pq$, where 
$p$ and $q$ are distinct primes.
However, they have not described the full spectrum explicitly for some of these graphs. In this paper, we give the full spectrum of the distance matrix of the power graph of these groups. 
The main aim of this paper is to study the distance spectra of the enhanced power graphs of finite groups. In Section \ref{sec:Prelim}, we develop some backgrounds and state some earlier known results, which are crucial for the forthcoming sections. In Section  \ref{sec:cyclic-dihedral-dicyclic}, we study the distance spectra of enhanced power graphs of groups of order $pq$, dihedral groups, and dicyclic groups. We conclude the paper by considering certain classes of finite abelian groups viz. 
$\El(p^n) \times \El(q^m)$ and 
$\El(p^n) \times \mathbb{Z}_m$ in Section \ref{sec:pnqm} and \ref{sec:finiteabelgrp}. 

\section{Preliminaries} 
\label{sec:Prelim}

For the convenience of the reader, we start by recalling some basic
definitions and notations. 

\begin{enumerate}
\item Throughout the paper, we consider $G$ as a finite group and denote its identity element by $e$.  $|G|$ denotes the cardinality of the set $G.$ For a prime $p,$ a group $G$ is said to be a $p$-group if $|G|=p^{r}$ for some $r\in \mathbb{N}.$ If
every element of $G$ is of order $p$ we call $G$ to be an \emph{elementary abelian $p$-group} and the elementary abelian group of order $p^n$ is denoted by $\El(p^n)$. We use $G^*$ to denote the set $G \backslash \{e\}$.

    \item Let $\Gamma$ be a graph with vertex set $V(\Gamma)$ and edge set $E(\Gamma)$. 
For a vertex $u$, we denote by $N(u)$, the set of vertices adjacent to $u$. A vertex is called a 
\emph{dominating vertex} if it is adjacent to every other
vertex in the graph. A \emph{path} in a graph is a finite sequence of edges which joins a sequence of distinct vertices. A graph is said to be \emph{connected} if there exists a path between for any pair of vertices $u$ and $v.$ The distance between two vertices $u$ and $v$ is the length of the shortest path between them and it is denoted by $d(u, v).$ For a graph $\Gamma,$ $\text{diam}(\Gamma)= \max_{u, v \in V} d(u, v)$ is called the \emph{diameter} of $\Gamma$. 

\item All vectors are considered to be column vectors unless stated otherwise. We use $\0_n$ and $\1_n$ to denote the zero vector and the all ones vector in $\rr^{n}$. The identity and the all ones matrix of order $n$ are denoted by $I_{n}$ and $J_{n}$, respectively. We use $J_{m,n}$ to denote the all ones matrix of order $m \times n$ and $O$ to denote the zero matrix of appropriate order. The Kronecker product of matrices is represented by $\otimes$.

\item A partition of a set $X$ is a set of non-empty subsets of $X$ such that every element of $X$ is in exactly one of these subsets.

\item Suppose $A$ is a real symmetric matrix whose rows and columns are indexed by $X=[n]$. Let $\{ X_1, X_2,\dots, X_m \}$ be a partition of $X$. 
Let $A$ be partitioned according to $\{X_1, \dots, X_m \}$, that is 
$$A=
 	 \left[
 	\begin{array}{ccc}
 	A_{1,1} & \cdots  & A_{1,m} \\
  \vdots & & \vdots \\
    A_{m,1} & \cdots  & A_{m,m} 
 	\end{array}
 	\right],
 	  $$
    where $A_{i,j}$ denotes the block submatrix of $A$ formed by rows in $X_i$ and the columns in $X_j$. Let $t_{i,j}$ denote the average row sum of $A_{i,j}.$ Then the matrix $T=(t_{i,j})$ is called the \textit{quotient matrix} of $A$ w.r.t the given partition (see Andries \cite{Andries2010}). 
    If the row sum of each block $A_{i,j}$ is constant then the partition $\{ X_1, X_2,\dots, X_m \}$ is called equitable.

\item Suppose $\Gamma$ is a simple graph with vertex set $V=\{1,2,\dots,n\}$. The adjacency matrix $A:=(a_{ij})$ of $\Gamma$ is an $n \times n$ matrix with 

\[a_{ij}:= \begin{cases}
1 &~\mbox{if}~$i$~\mbox{and}~$j$~\mbox{are adjacent}\\
0&~\mbox{otherwise}.
\end{cases}\] If a partition $V(\Gamma)=V_1 \cup \dotsc \cup V_m$ of the vertex set $V(\Gamma)$ is \emph{equitable} w.r.t the adjacency matrix, then it is easy to see that
for each $i \in [m]$ and for all $u,v \in V_i$, $|N(u) \cap V_j| = |N(v) \cap V_j|$, for all $j\in [m]$. Define $t_{ij} = |N(u) \cap V_j|$, then the quotient matrix corresponding to this equitable partition is $T=(t_{ij})$. Throughout this paper, we use the above discussion as the definition of an equitable partition of a graph.

\item For a simple graph $\Gamma$ with vertex set $V=\{1,2,\dots,n\}$, the matrix $D(\Gamma):=(d_{ij})$ is called the distance matrix of $\Gamma$, where 
\[d_{ij}:= \begin{cases}
0 &~\mbox{if}~i =j,\\
d(i,j)&~\mbox{otherwise}.
\end{cases}\]
Suppose $\text{diam}(\Gamma) \leq 2$ and let $\{V_1, \dotsc,V_m\}$ be an equitable partition of $\Gamma$. Let $D$ be partitioned as $D=(D_{ij})$, 
    where $D_{i,j}$ denotes the block submatrix of $D$ formed by rows in $V_i$ and the columns in $V_j$. It is easy to see that row sum of each block $D_{i,j}$ is constant. For each $i,j \in [m]$, if $t^D_{ij}$ denotes the constant row sum of $D_{i,j}$, then
    \begin{equation}\label{eqn:T}
    t^D_{ij}:= \begin{cases}
2|V_i|-2-|N(u) \cap V_i|&~\mbox{if}~i= j\\
2|V_j|-|N(u) \cap V_j|&~\mbox{otherwise.}  
\end{cases}
\end{equation}
    In this paper, the matrix $T^D:=(t^D_{ij})$ is called the \emph{distance quotient matrix} of $\Gamma$.

\item For a matrix $B$, the characteristic polynomial is denoted by $\phi(B,x)$. We denote the characteristic polynomial of the adjacency and distance matrix of a  graph $\Gamma$ by $\phi_A(\Gamma,x)$ and $\phi_D(\Gamma,x)$, respectively.

\item For a graph $\Gamma$ with vertex set $V(\Gamma)=\{1,\dotsc,n\}$, the $\Gamma$-join of a set of graphs $\{\Gamma_i:1\leq i\leq n\}$ is the graph $\Gamma[\Gamma_1, \Gamma_2, \dots, \Gamma_n]$ with vertex set $V$ and edge set $E$, where
\[V:=\{(i,u): 1\leq i \leq n, u \in V(\Gamma_i)\}\]
\[E:=\{(i,u)(j,v): \mbox{if}~ij \in E(\Gamma)~\mbox{or else}~ i=j~\mbox{and}~uv \in E(\Gamma_i)\}.\]
Note that the graph $\Gamma[\Gamma_1,\dotsc,\Gamma_n]$ is obtained by replacing each vertex $i \in V(\Gamma)$ by $\Gamma_i$. For the sake of convenience, we assume $V(\Gamma[\Gamma_1, \Gamma_2, \dots, \Gamma_n])=\cup_{i=1}^{n}V(\Gamma_i)$. 
 
\end{enumerate}

We now state some basic results, which are useful in proving the main results of the paper.
\begin{theorem}{\rm \cite[Chapter 2]{Andries2010}}\label{Tdivides_ch_pol}
Suppose $\Gamma$ is a graph on $n$ vertices with diameter at most $2$ and let $V_1,\dotsc,V_m$ be an equitable partition of $V(\Gamma)$. Let $T^D$ be the matrix defined in (\ref{eqn:T}). Then $\phi(T^D,x)$ divides $\phi_D(\Gamma,x)$.
\end{theorem}

The following result appears in 
\cite[Theorem 2.1]{stevanovic-AMC-2002}. Since it is presented differently than what is required for our results, we state it below in a different way.

\begin{theorem}\label{thm:join}
Suppose $\Gamma$ is a connected graph with vertex set $\{1,\dotsc,p\}$ and diameter atmost $2$. If  $\Gamma_i$, $1 \leq i \leq p$ are all $r_i$ regular with diameter at most 2, then $V(\Gamma_1), \dotsc, V(\Gamma_p)$ is an equitable partition of $\Gamma[\Gamma_1,\dotsc,\Gamma_p]$. Let $T^D$ denote the matrix associated with this partition as defined in (\ref{eqn:T}). Then the characteristic polynomial of the distance matrix of $\Gamma[\Gamma_1,\dotsc,\Gamma_p]$ is 
\[\phi_D(\Gamma[\Gamma_1,\dotsc,\Gamma_p],x)= \phi(T^D,x) \prod_{i=1}^{p} \frac{\phi_D(\Gamma_i,x)}{(x-2n_i+2+r_i)}.\]
Here, for each $i$, $n_i$ is the number of vertices in $\Gamma_i$.
\end{theorem}

\section{Groups of order $pq$, Dihedral groups, and Dicyclic groups}
\label{sec:cyclic-dihedral-dicyclic} 

In this section, we study the distance spectra of enhanced power graphs of groups of order $pq$, dihedral groups, and dicyclic groups.
We start with the non-abelian groups of order $pq$, where $p$ and $q$ are distinct primes. This is because the enhanced power graph of any abelian group of order $pq$ is complete. 
\begin{theorem}
\label{thm:nonyclicpq_Enhancedpower} 
Let $p,q$ be primes with $p<q$ and $G_{pq}$ denote the unique non-abelian group of order $pq$, when it exists. Then, the characteristic polynomial of the distance matrix of $\mathcal{G}_E(G_{pq})$ is
\begin{equation*} 
\begin{aligned} 
 \phi_D(\mathcal{G}_E(G_{pq}),x) & = (x+1)^{pq-q-2} (x+p)^{q-1} \bigg[x^3+x^2(-2pq+p+q+2)+ \\ & \hspace{6 mm}  x(-2pq^2-2pq+2q^2+2p+1) -(pq^2+pq-p-q^2)\bigg]. 
\end{aligned}
\end{equation*} 

\end{theorem}

\begin{proof}
By Sylow Theorem, there is a unique $q$-Sylow subgroup of $G_{pq}$, say $H_1$.
As $G_{pq}$ is non-cyclic, the number of p-Sylow subgroups
of $G_{pq}$ is $q$. Suppose $H_2,\dotsc,H_{q+1}$ are the p-Sylow subgroups
of $G_{pq}$. It is easy to see that $H_j$'s, $2 \leq j \leq q+1$ have a trivial intersection and hence no two non-identity elements of distinct $H_j$'s are adjacent in $\mathcal{G}_E(G_{pq})$. Moreover, since $G_{pq}$ is non-cyclic, no non-identity element of $H_1$ is  adjacent to a non-identity element of $H_j$ in $\mathcal{G}_E(G_{pq})$, for all $2 \leq j \leq q+1$. Finally, since each $H_j$, $1 \leq j \leq q+1$ induces a complete subgraph in $\mathcal{G}_E(G_{pq})$, we have
 \begin{equation}\label{EPG:pq:joinform}
     \mathcal{G}_E(G_{pq}) \cong K_{1,q+1}[K_1, K_{q-1}, \underbrace{K_{p-1},\dotsc,K_{p-1}}_{q} ].
 \end{equation}
 Suppose $T^D$ is the distance quotient matrix of $\mathcal{G}_E(G_{pq})$ corresponding to the equitable partition given in \eqref{EPG:pq:joinform}. Then 
 \[T^D= \left[\begin{array}{cccc}
     0& q-1 & (p-1)\1_q' \\
     1 & q-2 & (2p-2)\1_q' \\
     \1_q & (2q-2)\1_q & 2(p-1)J_{q}-pI_{q}
\end{array}\right].\]
Using Theorem \ref{thm:join}, we get 
\begin{equation} \label{EPG:pq:charpoly}
    \begin{aligned}
    \phi_D(\mathcal{G}_E(G_{pq}),x) &= \phi(T^D,x) \frac{\phi_D(K_{q-1},x)}{(x-(q-2))} \prod_{i=1}^{q} \frac{\phi_D(K_{p-1},x)}{(x-(p-2))}  \\
    &=\phi(T^D,x) (x+1)^{q-2} (x+1)^{(p-2)q}\\
    &=\phi(T^D,x)  (x+1)^{pq-q-2}.
    \end{aligned}
\end{equation}
We now determine the characteristic polynomial of $T^D$. We know that 
\begin{equation}
    \begin{aligned}
    \phi(T^D,x)&=\det(xI_{q+2}-T^D)\\ &= \left|\begin{array}{cccc}
     x & 1-q &  (1-p)\1_q' \\
     -1 & x-(q-2) & (2-2p)\1_q' \\
     -\1_q & (2-2q)\1_q & (x+p)I_q-2(p-1)J_q
\end{array}\right|. 
\end{aligned}
\end{equation} 
Using elementary row and column operations, we have

\begin{equation}\label{EPG:pq:Tcharpoly}
    \begin{aligned}
    \phi(T^D,x)
&= \left|\begin{array}{cccc}
     x & 1-q &  (1-p)\1_q' \\
     -1 & x-(q-2) & (2-2p)\1_q' \\
     0 & -(x+q)\1_q & (x+p)I_q
\end{array}\right| \\
& = x \left|\begin{array}{cc}
     x-q+2 & (2-2p)\1_q' \\
     -(x+q)\1_q & (x+p)I_q 
\end{array}\right| +
\left|\begin{array}{cc}
     1-q & (1-p)\1_q' \\
     -(x+q)\1_q & (x+p)I_q 
\end{array}\right|
\\
& = x (x-q+2)(x+p)^q+xq(2-2p)(x+q)(x+p)^{q-1} +(1-q)(x+p)^q \\ & \hspace{4 mm} +q(x+q)(1-p)(x+p)^{q-1}\\
& = (x+p)^{q-1} [x^3+x^2(-2pq+p+q+2)+x(-2pq^2-2pq+2q^2+2p+1)- \\
 & \hspace{6 mm} (pq^2+pq-p-q^2)]  .
\end{aligned}
\end{equation} 
From \eqref{EPG:pq:charpoly} and \eqref{EPG:pq:Tcharpoly}, the proof is complete. 
\end{proof}

\begin{corollary}
\label{cor:det_pqgrp}
Let $p,q$ be primes with $p<q$ and $G_{pq}$ denote the unique non-abelian group of order $pq$, when it exists. The distance matrix $D$ of $\mathcal{G}_E(G_{pq})$ is non-singular. 
\end{corollary}

\begin{proof}
 As $q>1$, we have $q^2+q-1>q^2$ and hence $p(q^2+q-1) > q^2.$ Therefore, setting $x=0$ in Theorem \ref{thm:nonyclicpq_Enhancedpower}, we see that $$\det(D)=p^{q-1}[p(q^2+q-1)-q^2],$$
 which is clearly non-zero. The proof is complete.  
\end{proof}

By \cite[Theorem 28]{AalipourcameronEJC} the power graph and the enhanced power graph for a finite group $G$ are same if and only if every cyclic subgroup of $G$ has prime power order. Thus, for a non-cyclic group $G$ of order $pq$, the power graph and the enhanced power graph are the same and hence  $\mathcal{P}(G_{pq}) $ and $\mathcal{G}_E(G_{pq})$ have identical distance spectra. 
Next, we compute the distance spectra of the enhanced power graph of the dihedral group $D_{2n}$. For $n \geq 3$, the \emph{dihedral group} $D_{2n}$ of order $2n$ is defined
as
$$D_{2n}=\langle a,b: a^n=b^2=e, ab=ba^{-1} \rangle .$$ 
\begin{theorem}
\label{thm:enhancedpower-dihedral}
For any positive integer $n$, the characteristic polynomial of the distance matrix of the enhanced power graph of $D_{2n}$ is
 $$\phi_D(\mathcal{G}_E(D_{2n}),x) = (x+2)^{n-1}(x+1)^{n-2}(x^3 - (3n-4)x^2-(2n^2+4n-5)x-n^2-2n+2).$$
\end{theorem}

\begin{proof}
It is easy to verify that 
$$V(\mathcal{G}_E(D_{2n}))= \langle e \rangle \cup \langle a \rangle^* \cup \langle ab \rangle^* \cup \dotsc \cup \langle a^{n-1}b \rangle^*$$ is an equitable partition for the vertex set of $\mathcal{G}_E(D_{2n})$. This yields,
 $$\mathcal{G}_E(D_{2n}) \cong K_{1,n+1}[K_1, K_{n-1}, \underbrace{K_1, \dotsc, K_1}_{n}].$$  By Theorem \ref{thm:join}, we get
$$\phi_D(\mathcal{G}_E(D_{2n}),x) = \phi(T^D,x) (x+1)^{2n-n-2},  $$
where \[T^D= \left[\begin{array}{cccc}
     0& n-1 & \1'_n \\
     1 & n-2 & \1'_n \\
     \1_n & (2n-2)\1_n & 2J_{n}-2I_{n}
\end{array}\right].\]
 Proceeding similarly as in the proof of Theorem \ref{thm:nonyclicpq_Enhancedpower}, we compute 
$$\phi(T^D,x)=(x+2)^{n-1}\big(x^3 - (3n-4)x^2-(2n^2+4n-5)x-n^2-2n+2\big).$$
This completes the proof.
\end{proof}
The following is an immediate consequence of theorem \ref{thm:enhancedpower-dihedral}.
\begin{corollary}
    The distance matrix of the enhanced power graph of $D_{2n}$ is non-singular.
\end{corollary}
By $\mathcal{P}^*(G)$, we denote the induced subgraph of $\mathcal{P}(G)$ on $G \setminus \{e\}.$ This is also known as \emph{proper power graph} of $G$.
Suppose $\mathbb{Z}_n$ denotes the cyclic group of order $n$. In the following result, we compute the characteristic polynomial of the distance matrix of $\mathcal{P}(D_{2n})$ in terms of the characteristic
polynomials of the distance matrix of $\mathcal{P}(\mathbb{Z}_n)$ and $\mathcal{P}^*(\mathbb{Z}_n)$. This is analogous to \cite[Theorem 2.2]{CHATTOPADHYAY2018730}.

\begin{theorem}
\label{thm:power-dihedral}
For any positive integer $n$, the characteristic polynomial of the distance matrix of the power graph of $D_{2n}$ is
\begin{small}
\begin{equation*}
    \phi_D(\mathcal{P}(D_{2n}),x) = (x+2)^{n-1}(((4n+1)x+2(n+1))\phi_D( \mathcal{P}(\mathbb{Z}_n),x)-n(2x+1)^2 \phi_D( \mathcal{P}^*(\mathbb{Z}_n),x)).
\end{equation*}
\end{small} 
\end{theorem}
\begin{proof}
From \cite{Mirzargar}, we first note that $\mathcal{P}(D_{2n})$ is a union of $\mathcal{P}(\mathbb{Z}_n)$ and $n$ copies of $K_2$ that share in the identity element of $D_{2n}$. Thus, we have
\begin{equation}\label{D2n}
    \mathcal{P}(D_{2n}) \cong K_{1,n+1}[K_1,\underbrace{K_1,\dotsc,K_1}_{n},\mathcal{P}^*(\mathbb{Z}_n)].
\end{equation}
From \cite[Theorem 2.2]{mehranian-gholami-ashrafi-IJGT}, we know that 
\begin{equation}\label{Zn}
    \mathcal{P}^*(\mathbb{Z}_n) \cong \Gamma_1[K_{\phi(n)}, K_{\phi(d_1)}, \dots, K_{\phi(d_t)}],
\end{equation}
where $\Gamma_1:=K_1+\Delta_n$ and $\Delta_n$ is the graph with vertex set $V(\Delta_n)=\{d_i:1,n\neq d_i, d_i | n, 1 \leq i \leq t \}$ and edge set 
$E(\Delta_n)=\{ d_id_j: d_i|d_j, 1 \leq i <j \leq t \}$. Combining (\ref{D2n}) and (\ref{Zn}), we have
\begin{equation}\label{PG:D2n:joinform}
    \begin{aligned}
    \mathcal{P}(D_{2n}) &\cong K_{1,n+1}[K_1,\underbrace{K_1,\dotsc,K_1}_{n},\Gamma_1[K_{\phi(n)}, K_{\phi(d_1)}, \dots, K_{\phi(d_t)}]]\\
    &\cong \Gamma_2[K_1,\underbrace{K_1,\dotsc,K_1}_{n},K_{\phi(n)}, K_{\phi(d_1)}, \dots, K_{\phi(d_t)}],
    \end{aligned}
\end{equation}
where $\Gamma_2$ is the graph obtained by identifying the lower last pendant vertex of $K_{1,n+1}$ with the vertex corresponding to $K_1$ in $\Gamma_1$. Here, we are looking at $K_{1,n+1}$ as a graph with a vertex on the left attached to $n+1$ vertices on its right. Now, by Theorem \ref{thm:join} 

\begin{equation}\label{PG:D2n:Dcharpoly}
    \begin{aligned}
\phi_D( \mathcal{P}(D_{2n}),x) 
&=\phi(T^D,x) (x+1)^{\phi(n)-1} \prod_{i=1}^t (x+1)^{\phi(d_i)-1}   \\
   &=\phi(T^D,x)  (x+1)^{\phi(n)-1} (x+1)^{n-\phi(1)-\phi(n)-t}\\
   &=\phi(T^D,x)  (x+1)^{n-t-2}.
    \end{aligned}
\end{equation}
Let $T^{D}_{\mathbb{Z}}$ be the distance quotient matrix of $\mathcal{P}^*(\mathbb{Z}_n)$ corresponding to the equitable partition given by $\mathcal{P}^*(\mathbb{Z}_n) \cong \Gamma_1[K_{\phi(n)}, K_{\phi(d_1)}, \dots, K_{\phi(d_t)}]$. Suppose
$T^{D}_{\mathbb{Z}}(1)$ denotes the first row of $T^{D}_{\mathbb{Z}}$. Then
$T^{D}_{\mathbb{Z}}(1)= (\phi(n)-1, \phi(d_1) , \phi(d_2) , \cdots, \phi(d_t)  )$. Let $T^D$ be the distance quotient matrix of $\mathcal{P}(D_{2n})$ for the partition given in \eqref{PG:D2n:joinform}. Then
\[T^D= \left[\begin{array}{ccccccccccc}
     0& \1_n' & T^{D}_{\mathbb{Z}}(1)+e_1' \\
     \1_n& 2(J_n-I_n)& 2\1_n (T^{D}_{\mathbb{Z}}(1)+e_1')\\
     \1_{t+1}& 2 J_{t+1,n} & T^{D}_{\mathbb{Z}}\\
\end{array}\right],\]
where $e_1=(1,0,\dotsc,0)' \in \rr^{t+1}$. Now, the characteristic polynomial of $T^D$ is
\begin{equation*}
    \begin{aligned}
    \phi(T^D,x) &= \det(xI_{n+t+2}-T^D)\\
    &= \left|\begin{array}{ccccccccccc}
     x& -\1_n' & -T^{D}_{\mathbb{Z}}(1)-e_1' \\
     -\1_n& (x+2)I_n-2J_n& -2\1_n (T^{D}_{\mathbb{Z}}(1)+e_1')\\
     -\1_{t+1}& -2J_{t+1,n}& xI_{t+1}-T^{D}_{\mathbb{Z}}\\
\end{array}\right|.
    \end{aligned}
\end{equation*}
Subtracting twice the first row from the second to $(n+1)^{\rm th}$ row and then subtracting twice the first column from the second to $(n+1)^{\rm th}$ column, we get
\begin{equation*}
    \begin{aligned}
    \phi(T^D,x) 
&= \left|\begin{array}{ccccccccccc}
     x& -(2x+1)\1_n' & -T^{D}_{\mathbb{Z}}(1)-e_1' \\
     -(2x+1) \1_n& (x+2)I_n+(4x+2)J_n& O\\
     -\1_{t+1}& O& xI_{t+1}-T^{D}_{\mathbb{Z}}\\
\end{array}\right|\\
&= \begin{small}
\left|\begin{array}{ccccccccccc}
     x& -(2x+1)\1_{n-1}' & -(2x+1) & -T^{D}_{\mathbb{Z}}(1)-e_1' \\
     -(2x+1)\1_{n-1}& (x+2)I_{n-1}+(4x+2)J_{n-1}& (4x+2)\1_{n-1}& O\\
     -(2x+1)& (4x+2)\1'_{n-1}& (5x+4)& \0_{t+1}'\\
     -\1_{t+1}& O& \0_{t+1} & xI_{t+1}-T^{D}_{\mathbb{Z}}\\
\end{array}\right|.
\end{small}\\
\end{aligned}
\end{equation*}
Subtracting the $(n+1)^{\rm th}$ row from the second to $n^{\rm th}$ row
and then adding the sum of the second to $n^{\rm th}$ column to the $(n+1)^{\rm th}$ column, we have
\begin{equation*}
    \begin{aligned}
    \phi(T^D,x) 
&= \left|\begin{array}{ccccccccccc}
     x& -(2x+1)\1_{n-1} & -n(2x+1) & -T^{D}_{\mathbb{Z}}(1)-e_1' \\
     \0_{n-1}& (x+2)I_{n-1}& \0_{n-1} & O\\
     -(2x+1)& (4x+2)\1'_{n-1}& (4n+1)x+2(n+1)& \0'_{t+1}\\
     -\1_{t+1}& O& \0_{t+1} & xI_{t+1}-T^{D}_{\mathbb{Z}}\\
\end{array}\right|\\
&=  ((4n+1)x+2(n+1)) \left|\begin{array}{ccccccccccc}
     x& -(2x+1)\1_{n-1}  & -T^{D}_{\mathbb{Z}}(1)-e_1' \\
     \0_{n-1}& (x+2)I_{n-1} & O\\
     -\1_{t+1}& O & xI_{t+1}-T^{D}_{\mathbb{Z}}\\
\end{array}\right|\\& ~~~ - n(2x+1)^2 (x+2)^{n-1}\phi(T^{D}_{\mathbb{Z}},x).
    \end{aligned}
\end{equation*}
In the first determinant, subtracting the $(n+1)^{\rm th}$ row from the first row  and then adding the first column to the $(n+1)^{\rm th}$ column, we get
\begin{equation}\label{PG:D2n:Tcharpoly}
    \begin{aligned}
      \phi(T^D,x) 
&= ((4n+1)x+2(n+1))\begin{small}
\left|\begin{array}{ccccccccccc}
     x+1& -(2x+1)\1_{n-1}  & \0_{t+1}' \\
     \0_{n-1}& (x+2)I_{n-1} & O\\
     -\1_{t+1}& O & xI_{t+1}-(T^{D}_{\mathbb{Z}}+\1_{t+1} e_1')\\
\end{array}\right|\end{small}\\
&~~~ - n(2x+1)^2 (x+2)^{n-1}\phi(T^{D}_{\mathbb{Z}},x)\\
&= (x+2)^{n-1}((x+1)((4n+1)x+2(n+1))\det(xI-(T^{D}_{\mathbb{Z}}+\1_{t+1} e_1'))
\\ &~~~- n(2x+1)^2 \phi(T^{D}_{\mathbb{Z}},x)).
    \end{aligned}
\end{equation}
Note that, $T^{D}_{\mathbb{Z}}+\1_{t+1} e_1'$ is the quotient matrix of $\mathcal{P}(\mathbb{Z}_n)$ corresponding to the equitable partition given by the expression $\mathcal{P}(\mathbb{Z}_n) \cong \Gamma_1[K_{\phi(n)+1}, K_{\phi(d_1)}, \dots, K_{\phi(d_t)}]$. Thus, from (\ref{PG:D2n:Dcharpoly}) and (\ref{PG:D2n:Tcharpoly}), we conclude that  
\begin{multline*}
    \phi_D( \mathcal{P}(D_{2n}),x)\\
       =(x+2)^{n-1}(((4n+1)x+2(n+1))\phi_D( \mathcal{P}(\mathbb{Z}_n),x)-n(2x+1)^2 \phi_D( \mathcal{P}^*(\mathbb{Z}_n),x)).    
\end{multline*}
The proof is complete.
\end{proof}

We now compute the distance spectra of the enhanced power graph of the dicyclic group. For $n \geq 3$, the dicyclic group $Dic_{4n}$ of order $4n$ is defined as
$$Dic_{4n}=\langle a,x: a^{2n}=e, x^2=a^n, ax=xa^{-1}  \rangle .$$ 

\begin{theorem}
\label{thm:enhancedpower-dicyclic}
For any positive integer $n$, the characteristic polynomial of the distance matrix of the enhanced power graph of $Dic_{4n}$ is
 $$\phi_D(\mathcal{G}_E(Dic_{4n}),x) = (x+1)^{3n-2} (x+3)^{n-1}(x^3-(6n-5)x^2-(8n^2+4n-7)x-(6n-3)) .$$
\end{theorem}

\begin{proof}
There are two dominating vertices in $\mathcal{G}_E(Dic_{4n})$, one is the identity $e$ and the other is element $a^n$. Since $a^{2n}=e$ and $(a^ix)^2=x^2$, for $1 \leq i \leq n$
 $$\mathcal{G}_E(Dic_{4n})\cong K_{1,n+1}[K_2,K_{2n-2}, \underbrace{K_{2}, \dotsc, K_{2}}_{n} ].$$ 
Suppose $T^D$ is the distance quotient matrix of $\mathcal{G}_E(Dic_{4n})$ corresponding to the equitable partition given in \eqref{EPG:pq:joinform}. Then  
\[T^D= \left[\begin{array}{cccc}
     1& 2n-2 & 2\1_n' \\
     2 & 2n-3 & 4\1_n' \\
     2\1_n & 2(2n-2)\1_n & 4J_{n}-3I_{n}
\end{array}\right].\]
By Theorem \ref{thm:join}, we get
$$\phi_D(\mathcal{G}_E(Dic_{4n}),x) = \phi(T^D,x) (x+1)^{3n-2}. $$
Proceeding similar to the proof of Theorem \ref{thm:nonyclicpq_Enhancedpower}, the characteristic polynomial of $T^D$ can be easily computed. This completes the proof. 
\end{proof}
As a consequence of the above theorem, the distance matrix of the enhanced power graph of $Dic_{4n}$ is non-singular. Suppose $n$ is a power of $2$. From Theorem \ref{thm:enhancedpower-dicyclic}, we immediately have the following: $$\phi_D(\mathcal{P}(Dic_{4n}),x) = (x+1)^{3n-2} (x+3)^{n-1}(x^3-(6n-5)x^2-(8n^2+4n-7)x-(6n-3)) .$$
Wani and Srivastsav \cite[Proposition 2.12]{wani2020} also found out the distance quotient matrix of the graph $\mathcal{P}(Dic_{4n})$ but the $(n+1,n+2)$-th entry of their matrix is $2$, which is not correct, as seen from the proof of Theorem \ref{thm:enhancedpower-dicyclic}.

\section{The finite abelian group $\El(p^n) \times \El(q^m)$}\label{sec:pnqm}

 In this section, we are interested in studying the spectral properties of the power graph and the enhanced power graph of certain non-cyclic finite abelian groups. Kelarev and Quinn gave the structure of the power graphs of
all finite abelian groups in \cite{combinatorialpropertyandpowergraphsofgroupskq1}. We further study the structure of power and enhanced power graphs of these groups and obtain a precise representation in terms of equitable partitions.  
Suppose $p$ and $q$ are distinct primes.
We give two equitable partitions for the power graph of $\El(p^n) \times \El(q^m)$.  These partitions are equitable partitions for the enhanced power graph of $\El(p^n) \times \El(q^m)$ as well. The later is proved before section \ref{sec:EPG:pnqm:AM}. We recall that $e$ represents the identity element of a group. 
\begin{lemma}\label{PG:pnqm:partition1}
Define $V_1:= \{(e,e)\}$,  $V_2:= \{ (a, e): a \in \El(p^n)^* \}$,
$V_3:= \{ (a,b): a \in \El(p^n)^*,  b \in \El(q^m)^*\}$,  and 
$V_4:= \{ (e,b): b \in \El(q^m)^*\}$.
Then $\{V_1, V_2, V_3, V_4\}$  is an equitable partition for $\mathcal{P}(\El(p^n) \times \El(q^m))$.
\end{lemma}
\begin{proof}
Let $G=\El(p^n) \times \El(q^m)$. Since the sets $V_1, V_2, V_3$ and $V_4$ are mutually exclusive with cardinalities $1$, $p^n-1$, $(p^n-1)(q^m-1)$ and $q^m-1$,  respectively, $\{V_1,V_2,V_3,V_4\}$ forms a partition for $V(\mathcal{P}(G))$. Suppose $1 \leq i \leq 4$. We need to show that for each $i$ and for all $u,v \in V_i$, 
$$|N(u) \cap V_j| = |N(v) \cap V_j|,~\mbox{for}~ j=1,2,3,4.$$
For $i=1$, as $(e,e)$ is a dominating vertex, 
this readily holds.
Observe that for any $(a_1,b_1),(a_2,b_2) \in G$,
\begin{equation}\label{PG:pnqm:observation}
    (a_2,b_2) \in \langle (a_1,b_1) \rangle,~\mbox{if and only if}~ a_2 \in \langle a_1 \rangle~\mbox{and}~b_2 \in \langle b_1 \rangle.
\end{equation}
Using this observation, we have the following
\begin{enumerate}
    \item For any $u \in V_2$,
    \[|N(u) \cap V_1| = 1,~ |N(u) \cap V_2| = p-2,\]
    \[|N(u) \cap V_3| =(p-1)(q^m-1)~\mbox{and}~  |N(u)\cap V_4|=0.\]
\item For any $u \in V_3$,
    \[|N(u) \cap V_1| = 1,~ |N(u) \cap V_2| = p-1,\]
    \[|N(u) \cap V_3| = (p-1)(q-1)-1~\mbox{and}~  |N(u)\cap V_4|=q-1.\]
\item For any $u \in V_4$,
    \[|N(u) \cap V_1| = 1,~ |N(u) \cap V_2| = 0,\]
    \[|N(u) \cap V_3| = (p^n-1)(q-1)~\mbox{and}~  |N(u)\cap V_4|=q-2.\]
\end{enumerate}
Hence, $\{V_1, V_2, V_3, V_4\}$  is an equitable partition for $\mathcal{P}(G)$. The proof is complete.
\end{proof}

\begin{lemma}\label{PG:pnqm:partition2}
Suppose $\alpha = \frac{p^n-1}{p-1}$, and $\beta = \frac{q^m-1}{q-1}$. Let $\langle (a_1,e) \rangle ,\dots, \langle (a_{\alpha},e)\rangle $ be distinct subgroups of order $p$ and $\langle (e,b_1) \rangle ,\dots, \langle (e,b_{\beta}) \rangle $ be distinct subgroups of order $q$ in $\El(p^n) \times \El(q^m)$. Define $U_1:=\{(e, e)\} $, $V_i:= \langle (a_i,e) \rangle ^*$, $W_{i,j}:= \langle (a_i,b_j) \rangle \setminus (\langle (a_i,e) \rangle \cup \langle (e,b_j) \rangle )$ and $X_i:=\langle (e,b_i) \rangle ^*$, for $i\in[\alpha]$ and $j\in[\beta]$. Then  $$P:=\{U_1, V_1, \dotsc, V_\alpha, W_{1,1}, \dotsc , W_{1,\beta},  \dotsc,  W_{\alpha,1}, \dotsc , W_{\alpha,\beta}, X_1, \dotsc, X_\beta\}$$ forms an equitable partition for $\mathcal{P}(\El(p^n) \times \El(q^m))$.

\end{lemma}

\begin{proof}
Let $G=\El(p^n) \times \El(q^m)$. Note that the sets in $P$ are mutually exclusive with $|U_1|=1$, $|V_i|=p-1$, $|W_{i,j}|=(p-1)(q-1)$ and $|X_i|=q-1$,  for $i\in [\alpha]$ and $j\in [\beta]$. Therefore, $P$ forms a partition for $V(\mathcal{P}(G))$. We need to show that $P$ is an equitable partition. 
Since, $(e,e)$ is a dominating vertex, $$|N((e,e)) \cap Y| = |Y|,~\mbox{for}~ Y\in P
\setminus\{U_1\}.$$
Using \eqref{PG:pnqm:observation}, we have the following
\begin{enumerate}
    \item Fix $i\in [\alpha]$ and let $u \in V_i$. Then for any $i' \in [\alpha]$ and $j \in [\beta],$
    \[|N(u) \cap U_1| = 1,~ |N(u) \cap V_{i'}|=\begin{cases}
  p-2 &~\mbox{if}~i'=i,\\
  0&~\mbox{otherwise,}
\end{cases}\]
    \[|N(u)\cap X_j|=0~\mbox{and}~|N(u) \cap W_{i',j}|=\begin{cases}
  (p-1)(q-1) &~\mbox{if}~i'=i,\\
  0&~\mbox{otherwise.}
\end{cases}\]

    \item Fix $i\in [\alpha]$, $j \in [\beta]$ and let $u \in W_{i,j}$. Then for any $i' \in [\alpha]$ and $j' \in [\beta],$
    \[|N(u) \cap U_1| = 1,~ |N(u) \cap V_{i'}|=\begin{cases}
  p-1 &~\mbox{if}~i'=i,\\
  0&~\mbox{otherwise,}
\end{cases}\]
    \[|N(u) \cap W_{i',j'}|=\begin{cases}
  (p-1)(q-1)-1 &~\mbox{if}~i'=i~\mbox{and}~j'=j,\\
  0&~\mbox{otherwise,}
\end{cases}\]
and
\[|N(u)\cap X_{j'}|=\begin{cases}
  q-1 &~\mbox{if}~j'=j,\\
  0&~\mbox{otherwise.}
\end{cases}\]

 \item Fix $j \in [\beta]$ and let $u \in X_{j}$. Then for any $i \in [\alpha]$ and $j' \in [\beta],$
    \[|N(u) \cap U_1| = 1,~ |N(u) \cap V_{i}|=0,\]
    \[|N(u) \cap W_{i,j'}|=\begin{cases}
  (p-1)(q-1) &~\mbox{if}~j'=j,\\
  0&~\mbox{otherwise,}
\end{cases}\]
and
\[|N(u)\cap X_{j'}|=\begin{cases}
  q-2 &~\mbox{if}~j'=j,\\
  0&~\mbox{otherwise.}
\end{cases}\]
\end{enumerate} 
Hence, $P$  is an equitable partition for $\mathcal{P}(G)$ and the proof is complete.
\end{proof} 
From Lemma \ref{PG:pnqm:partition2}, the following result is immediate.
\begin{lemma}\label{PG:pnqm:joinform}
Let $G=\El(p^n) \times \El(q^m)$. Suppose $\alpha = \frac{p^n-1}{p-1}$ and $\beta = \frac{q^m-1}{q-1}$. Then,
\begin{small}
    \begin{equation*}
\mathcal{P}(G)\cong (K_1+\Gamma)[K_1,\underbrace{K_{(p-1)},\dots,K_{(p-1)}}_{\alpha},\underbrace{K_{(p-1)(q-1)},\dots,K_{(p-1)(q-1)}}_{\alpha \beta},\underbrace{K_{(q-1)},\dots,K_{(q-1)}}_{\beta}],  
    \end{equation*}
\end{small}where $\Gamma$ is the graph given in Figure \ref{fig}.
\end{lemma}
\begin{figure}[!h]
\centering
\begin{tikzpicture}[shorten >=1pt, auto, node distance=3cm, ultra thick,
   node_style/.style={circle,draw=black,fill=white !20!,font=\sffamily\Large\bfseries},
   edge_style/.style={draw=black, ultra thick}]
\node[label=left:\footnotesize$1$,draw] (1) at  (0,0) {};
\node[label=left:\footnotesize$2$,draw] (2) at  (0,-3.5) {};
\node (3) at  (0,-3.7) {};
\node (alpha-1) at  (0,-6.8) {}; 
\node[label=left:\footnotesize$\bf{\alpha}$,draw] (alpha) at  (0,-7) {}; 
\node[label=above:\footnotesize${\alpha+1}$,draw] (alpha+1) at  (6,1.1) {};
\node[label=right:\footnotesize$\alpha+2$,draw] (alpha+2) at  (6,0.35) {};
\node (alpha+3) at  (6,0.25) {};
\node (alpha+beta-1) at  (6,-0.9) {};
\node[label=right:\footnotesize$\alpha+\beta$,draw] (alpha+beta) at  (6,-1) {};
\node[label=above:\footnotesize$\alpha+\beta+1$,draw] (alpha+beta+1) at  (6,-2.4) {};
\node[label=above:\footnotesize$\alpha+\beta+2$,draw] (alpha+beta+2) at  (6,-3.15) {};
\node (alpha+beta+3) at  (6,-3.25) {};
\node (alpha+2beta-1) at  (6,-4.4) {};
\node[label=below:\footnotesize$\alpha+2\beta$,draw] (alpha+2beta) at  (6,-4.5) {};
\node[label=left:\footnotesize$\alpha+(\alpha-1)\beta+1$,draw] (alpha+alpha-1beta+1) at  (6,-5.9) {};
\node[label=right:\footnotesize$\alpha+(\alpha-1)\beta+2$,draw] (alpha+alpha-1beta+2) at  (6,-6.65) {};
\node (alpha+alpha-1beta+3) at  (6,-6.75) {};
\node (alpha+alphabeta-1) at  (6,-7.9) {};
\node[label=below:\footnotesize$\alpha+\alpha\beta$,draw] (alpha+alphabeta) at  (6,-8) {};
\node[label=right:\footnotesize$\alpha+\alpha\beta+1$,draw] (alpha+alphabeta+1) at  (12,0) {};
\node[label=right:\footnotesize$\alpha+\alpha\beta+2$,draw] (alpha+alphabeta+2) at  (12,-3.5) {};
\node (alpha+alphabeta+3) at  (12,-3.7) {};
\node (alpha+alphabeta+beta-1) at  (12,-6.8) {};
\node[label=right:\footnotesize$\alpha+\alpha\beta+\beta$,draw] (alpha+alphabeta+beta) at  (12,-7) {};
\draw  (1) to (alpha+1);
\draw  (1) to (alpha+2);
\draw  (1) to (alpha+beta);
\draw  (2) to (alpha+beta+1);
\draw  (2) to (alpha+beta+2);
\draw  (2) to (alpha+2beta);
\draw  (alpha) to (alpha+alpha-1beta+1);
\draw  (alpha) to (alpha+alpha-1beta+2);
\draw  (alpha) to (alpha+alphabeta);
\draw  (alpha+alphabeta+1) to (alpha+1);
\draw  (alpha+alphabeta+1) to (alpha+beta+1);
\draw  (alpha+alphabeta+1) to (alpha+alpha-1beta+1);
\draw  (alpha+alphabeta+2) to (alpha+2);
\draw  (alpha+alphabeta+2) to (alpha+beta+2);
\draw  (alpha+alphabeta+2) to (alpha+alpha-1beta+2);
\draw  (alpha+alphabeta+beta) to (alpha+beta);
\draw  (alpha+alphabeta+beta) to (alpha+2beta);
\draw  (alpha+alphabeta+beta) to (alpha+alphabeta);
\begin{scope}[dotted]
\draw  (3) to (alpha-1);  
\draw (alpha+3) to (alpha+beta-1);
\draw (alpha+beta+3) to (alpha+2beta-1);
\draw (alpha+alpha-1beta+3) to (alpha+alphabeta-1);
\draw (alpha+alphabeta+3) to (alpha+alphabeta+beta-1);
\end{scope}
\end{tikzpicture}
\caption{The graph $\Gamma$ related to the power graph of $\El(p^n) \times \El(q^m)$.} \label{fig}
\end{figure}
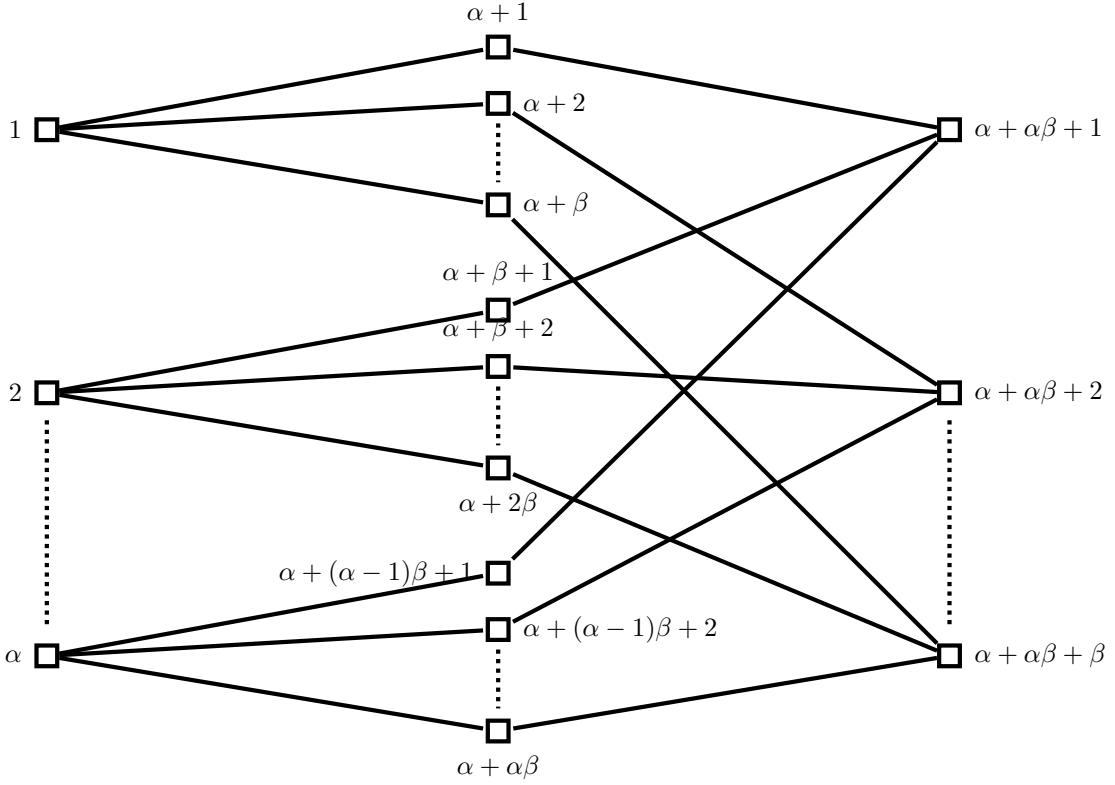
We now illustrate Lemma \ref{PG:pnqm:joinform} by the following example.
\begin{example}\rm
Suppose $p=2$, $q=3$ and $n=m=2$. Then  
\[V_1 = \langle (1,0,0,0) \rangle^*, V_2 = \langle (0,1,0,0) \rangle^*, V_3 = \langle (1,1,0,0) \rangle^*,\]
and 
\[X_1 = \langle (0,0,0,1) \rangle^*, X_2 = \langle (0,0,1,0) \rangle^*, X_3 = \langle (0,0,1,1) \rangle^* , X_4 = \langle (0,0,1,2) \rangle^*. \]
Therefore
\[W_{1,1} = \{(1, 0, 0, 1),(1, 0, 0, 2)\}, W_{1,2} = \{(1, 0, 1, 0), (1, 0, 2, 0)\}, \]
\[W_{1,3} = \{(1, 0, 1, 1), (1, 0, 2, 2)\}, W_{1,4} = \{(1, 0, 1, 2), (1, 0, 2, 1)\},\]
\[W_{2,1} = \{(0, 1, 0, 1), (0, 1, 0, 2)\}, W_{2,2} = \{(0, 1, 1, 0), (0, 1, 2, 0) \}, \]
\[W_{2,3} = \{(0, 1, 1, 1),(0, 1, 2, 2) \}, W_{2,4} = \{(0, 1, 1, 2),(0, 1, 2, 1)\},\]
\[W_{3,1} = \{(1, 1, 0, 1), (1, 1, 0, 2)\}, W_{3,2} = \{(1, 1, 1, 0), (1, 1, 2, 0) \}, \]
\[W_{3,3} = \{(1, 1, 1, 1), (1, 1, 2, 2) \},~\mbox{and}~ W_{3,4} = \{(1, 1, 1, 2), (1, 1, 2, 1)\}.\]
Hence the power graph of $\El(4) \times \El(9)$ is the graph in Figure \ref{PG:p2q2}. In Figure \ref{PG:p2q2}, each vertex  represents either the graph $K_1$ or $K_2$ and an edge between two graphs implies all the vertices of one graph are connected to all the vertices of the other graph.
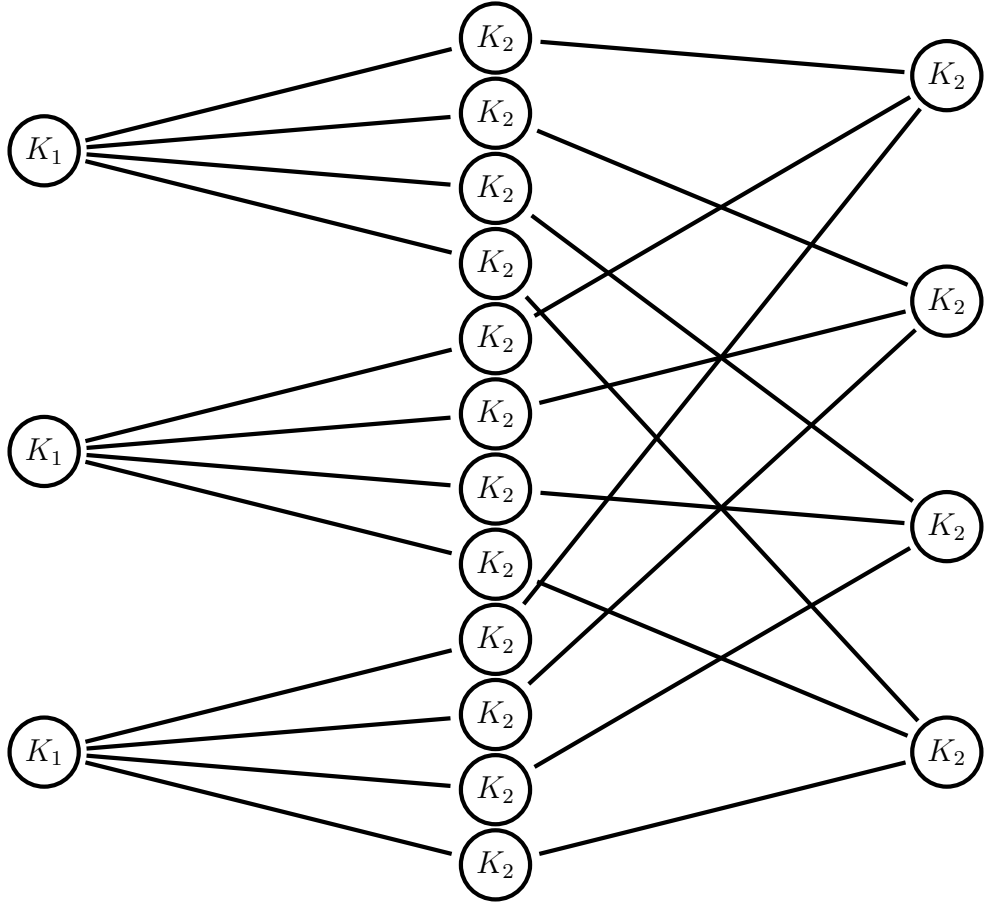
\begin{figure}[!h]
\centering
\begin{tikzpicture}[shorten >=1pt, auto, node distance=3cm, ultra thick,
   node_style/.style={circle,draw=black,fill=white !20!,font=\sffamily\Large\bfseries},
   edge_style/.style={draw=black, ultra thick}]
\node[vertex] (1) at  (0,0) {$K_1$};
\node[vertex] (2) at  (0,-4) {$K_1$};
\node[vertex] (alpha) at  (0,-8) {$K_1$};
\node[vertex] (alpha+1) at  (6,1.5) {$K_2$};
\node[vertex] (alpha+2) at  (6,0.5) {$K_2$};
\node[vertex] (alpha+3) at  (6,-0.5) {$K_2$};
\node[vertex] (alpha+beta) at  (6,-1.5) {$K_2$};

\node[vertex](alpha+beta+1) at  (6,-2.5) {$K_2$};
\node[vertex](alpha+beta+2) at  (6,-3.5) {$K_2$};
\node[vertex] (alpha+beta+3) at  (6,-4.5) {$K_2$};
\node[vertex] (alpha+2beta) at  (6,-5.5) {$K_2$};

\node[vertex](alpha+alpha-1beta+1) at  (6,-6.5) {$K_2$};
\node[vertex](alpha+alpha-1beta+2) at  (6,-7.5) {$K_2$};
\node[vertex] (alpha+alpha-1beta+3) at  (6,-8.5) {$K_2$};
\node[vertex](alpha+alphabeta) at  (6,-9.5) {$K_2$};

\node[vertex](alpha+alphabeta+1) at  (12,1) {$K_2$};
\node[vertex](alpha+alphabeta+2) at  (12,-2) {$K_2$};
\node[vertex] (alpha+alphabeta+3) at  (12,-5) {$K_2$};
\node[vertex] (alpha+alphabeta+beta) at  (12,-8) {$K_2$};
\draw  (1) to (alpha+1);
\draw  (1) to (alpha+2);
\draw  (1) to (alpha+3);
\draw  (1) to (alpha+beta);
\draw  (2) to (alpha+beta+1);
\draw  (2) to (alpha+beta+2);
\draw  (2) to (alpha+beta+3);
\draw  (2) to (alpha+2beta);
\draw  (alpha) to (alpha+alpha-1beta+1);
\draw  (alpha) to (alpha+alpha-1beta+2);
\draw  (alpha) to (alpha+alpha-1beta+3);
\draw  (alpha) to (alpha+alphabeta);
\draw  (alpha+alphabeta+1) to (alpha+1);
\draw  (alpha+alphabeta+1) to (alpha+beta+1);
\draw  (alpha+alphabeta+1) to (alpha+alpha-1beta+1);
\draw  (alpha+alphabeta+2) to (alpha+2);
\draw  (alpha+alphabeta+2) to (alpha+beta+2);
\draw  (alpha+alphabeta+2) to (alpha+alpha-1beta+2);
\draw  (alpha+alphabeta+3) to (alpha+3);
\draw  (alpha+alphabeta+3) to (alpha+beta+3);
\draw  (alpha+alphabeta+3) to (alpha+alpha-1beta+3);
\draw  (alpha+alphabeta+beta) to (alpha+beta);
\draw  (alpha+alphabeta+beta) to (alpha+2beta);
\draw  (alpha+alphabeta+beta) to (alpha+alphabeta);
\end{tikzpicture}
\caption{Proper power graph of $\El(4) \times \El(9)$.} \label{PG:p2q2}
\end{figure}
\end{example}

\subsection{Adjacency spectrum of $\mathcal{P}(\El(p^n) \times \El(q^m))$}
In the next lemma, we compute the quotient matrix of  $\mathcal{P}(\El(p^n) \times \El(q^m))$ corresponding to the equitable partitions given in Lemma \ref{PG:pnqm:partition1} and Lemma \ref{PG:pnqm:partition2}. As the proof of the lemma is direct, we omit it.
\begin{lemma}\label{PG:pnqm:T1T2}
The quotient matrices $T_1$ and $T_2$ of $\mathcal{P}( \El(p^n) \times \El(q^m))$ corresponding to the equitable partition given in Lemma \ref{PG:pnqm:partition1} and Lemma \ref{PG:pnqm:partition2}, respectively are
\begin{equation*}\label{PG:pnqm:T1}
T_1= \left[\begin{array}{cccc}
     0& p^n-1 &  (p^n-1)(q^m-1) &  q^m-1\\
     1& p-2 &  (p-1)(q^m-1) &  0\\
     1& p-1 &  (p-1)(q-1)-1 &  q-1\\
     1& 0 &  (p^n-1)(q-1) &  q-2\\
\end{array}
\right],
\end{equation*}
and 
\begin{equation} \label{PG:pnqm:T2}
  T_2=\left[ \begin{array}{cccc}
     0& (p-1)\1'_\alpha & (p-1)(q-1)\1'_{\alpha \beta} & (q-1) \1'_\beta \\
     \1_\alpha& (p-2)I_\alpha & (p-1)(q-1)I_\alpha \otimes \1'_\beta & O\\
     \1_{\alpha \beta} & (p-1) I_\alpha \otimes \1_\beta & ((p-1)(q-1)-1)I_{\alpha \beta} & (q-1)\1_\alpha \otimes I_\beta \\
     \1_\beta & O & (p-1)(q-1) \1'_\alpha \otimes I_\beta & (q-2) I_\beta
\end{array}
\right],
\end{equation}
where $\alpha = \frac{p^n-1}{p-1}$, and $\beta = \frac{q^m-1}{q-1}$.
\end{lemma}
We now compute the characteristic polynomial of $T_2$.
\begin{lemma}\label{PG:pnqm:T2charpoly}
Suppose $T_2$ is the matrix given in \eqref{PG:pnqm:T2}. Then
\[\phi(T_2,x) = \phi(T_1,x) (x-pq+p+q)^{(\alpha-1)(\beta-1)} f(x)^{\alpha-1}g(x)^{\beta-1},\]
where $\alpha = \frac{p^n-1}{p-1}$, $\beta = \frac{q^m-1}{q-1}$, and $f(x)$ and $g(x)$ are the characteristic polynomials of the matrices $\left[\begin{array}{cc}
     p-2&  (p-1)(q^m-1)\\
     p-1& (p-1)(q-1)-1
\end{array}\right]$ and $\left[\begin{array}{cc}
     (p-1)(q-1)-1&q-1  \\
     (p^n-1)(q-1)& q-2
\end{array}\right]$, respectively.
\end{lemma}
\begin{proof}
Suppose $(x_1,x_2,x_3,x_4)' \in \rr^{4}$ is an eigenvector of $T_1$ corresponding to an eigenvalue $\lambda$. Consider the vector $y=(x_1, x_2 \1_\alpha',x_3 \1_{\alpha \beta}',x_4\1_\beta')' \in \rr^{1+\alpha+\alpha\beta+\beta}$. Then, $y$ is an eigenvector of $T_2$ corresponding to the eigenvalue $\lambda$. Let $S_1$ be the set of eigenvectors of $T_2$ formed this way from the linear independent eigenvectors of $T_1$. Note that any linear dependence among the vectors in $S_1$ must provide an analogous dependence among the eigenvectors of $T_1$. Thus, we conclude that $S_1$ is linearly independent. 
Hence $\phi(T_1,x)$ divides $\phi(T_2,x)$. 
For $i \in \{1,\dotsc,\alpha-1\}$ and $j \in \{1,\dotsc,\beta-1\}$, we define the vectors $v^i:= (v^i_k)' \in \rr^{\alpha}$ and $w^j:= (w^j_k)' \in \rr^{\beta}$ as follows: 
\begin{equation*}\label{PG:pnqm:T2evec}
    v^i_k:= \begin{cases}
~~1&~\mbox{if}~k=i\\
-1&~\mbox{if}~k=\alpha\\
~~0&~\mbox{otherwise}
\end{cases}
~~~~~\mbox{and}~~~~~w^j_k:= \begin{cases}
~~1&~\mbox{if}~k=j\\
-1&~\mbox{if}~k=\beta\\
~~0&~\mbox{otherwise.}
\end{cases}
\end{equation*}
Now, for each $i \in \{1,\dotsc,\alpha-1\}$ and $j \in \{1,\dotsc,\beta-1\}$, 
\begin{equation*}
    \begin{aligned}
    T_2 \left[\begin{array}{c}
    0  \\
    \0_\alpha\\
    v^i\otimes w^j\\
    \0_\beta 
\end{array}
\right]
 &= \left[ \begin{array}{c}
     (p-1)(q-1)\1'_{\alpha \beta} ( v^i\otimes w^j)\\
     ((p-1)(q-1)I_\alpha \otimes \1'_\beta)( v^i\otimes w^j)\\
     ((p-1)(q-1)-1)I_{\alpha \beta}( v^i\otimes w^j)\\
     ((p-1)(q-1) \1'_\alpha \otimes I_\beta) ( v^i\otimes w^j)
\end{array}
\right]
=(pq-p-q)\left[ \begin{array}{c}
    0  \\
    \0_\alpha\\
    v^i\otimes w^j\\
    \0_\beta
\end{array}
\right].
    \end{aligned}
\end{equation*}
Since $S_2= \{(0 ,
    \0'_\alpha,
    v^i\otimes w^j,
    \0_\beta')': 1 \leq i \leq\alpha-1, 1 \leq j \leq \beta-1 \}$ is a linearly independent set, $pq-p-q$ is an eigenvalue of $T_2$ with multiplicity $(\alpha-1)(\beta-1)$. Clearly, $S_1 \cup S_2$ is a set of linearly independent eigenvectors of $T_2$. Now consider the following matrices
\[B=\left[\begin{array}{cc}
     p-2&  (p-1)(q^m-1)\\
     p-1& (p-1)(q-1)-1
\end{array}\right],~\mbox{and}~ C=\left[\begin{array}{cc}
     (p-1)(q-1)-1&q-1  \\
     (p^n-1)(q-1)& q-2
\end{array}\right].\]
Suppose $x=(x_1,x_2)$ is an eigenvector of $B$ corresponding to the eigenvalue $\lambda$. For each $i \in \{1,\dotsc,\alpha-1\}$, define $z^i:=(0,x_1v^i, x_2(v^i \otimes \1_\beta),\0_\beta)'$. Using the fact that $Bx=\lambda x$, we have for each $i$
\begin{equation*}
    \begin{aligned}
    T_2 z^i &= \left[\begin{array}{cc}
   x_1 (p-1)\1'_\alpha v^i+x_2(p-1)(q-1)\1'_{\alpha \beta}(v^i \otimes \1_\beta)\\
         x_1(p-2)I_\alpha v^i+x_2((p-1)(q-1)I_\alpha \otimes \1'_\beta)(v^i \otimes \1_\beta)\\
          x_1 ((p-1) I_\alpha \otimes \1_\beta) v^i+x_2((p-1)(q-1)-1)I_{\alpha \beta}(v^i \otimes \1_\beta)\\
           x_2((p-1)(q-1) \1'_\alpha \otimes I_\beta)(v^i \otimes \1_\beta)
    \end{array}\right]= \lambda z^i.
    \end{aligned}
\end{equation*}
Since $S_3:=\{z^i: i=1, \dotsc, \alpha-1\}$ is a linearly independent set of vectors, every eigenvalue of $B$ is an eigenvalue of $T_2$ with multiplicity $\alpha-1$. Similarly, suppose $x=(x_1,x_2)$ is an eigenvector of $C$ corresponding to the eigenvalue $\mu$. For each  $j \in \{1,\dotsc,\beta-1\}$, define $\widetilde{z}^j:=(0,\0_\alpha, x_1( \1_\alpha \otimes w^j ),x_2w^j)'$. Then
\begin{equation*}
    \begin{aligned}
    T_2 \widetilde{z}^j = \mu \widetilde{z}^j.
    \end{aligned}
\end{equation*}
Noting that $S_4=\{\widetilde{z}^j: j=1, \dotsc, \beta-1\}$ is a linearly independent set, we can conclude that every eigenvalue of $C$ is an eigenvalue of $T_2$ with multiplicity $\beta-1$. Suppose $f(x)$ and $g(x)$ are the characteristic polynomials of the matrices $B$ and $C$, respectively. Since 
$S_1 \cup S_2 \cup S_3 \cup S_4$  is linearly independent, we have 
\[\phi(T_2,x) = \phi(T_1,x) (x-pq+p+q)^{(\alpha-1)(\beta-1)} f(x)^{\alpha-1}g(x)^{\beta-1}.\]
This completes the proof.
\end{proof}

We are now ready to compute the characteristic polynomial of the adjacency matrix of the power graph of $\El(p^n) \times \El(q^m)$.
\begin{theorem}\label{PG:pnqm:AM:charpoly}.
Let $G=\El(p^n) \times \El(q^m)$. If $A$ is the adjacency matrix of $\mathcal{P}(G)$, then
\[\phi_A(\mathcal{P}(G),x)= \phi(T_1,x) (x+1)^{(p^nq^m-(\alpha +1)(\beta+1))} (x-pq+p+q)^{(\alpha-1)(\beta-1)} f(x)^{\alpha-1}g(x)^{\beta-1},\]
where $\alpha = \frac{p^n-1}{p-1}$, $\beta = \frac{q^m-1}{q-1}$, and $f(x)$ and $g(x)$ are the characteristic polynomials of the matrices $\left[\begin{array}{cc}
     p-2&  (p-1)(q^m-1)\\
     p-1& (p-1)(q-1)-1
\end{array}\right]$ and $\left[\begin{array}{cc}
     (p-1)(q-1)-1&q-1  \\
     (p^n-1)(q-1)& q-2
\end{array}\right]$, respectively.
\end{theorem}
\begin{proof}
Since the partition corresponding to $T_2$ is an equitable
partition for $\mathcal{P}(G)$, by Theorem 4 in 
\cite{schwenk74}, $\phi(T_2,x)$ divides $\phi_A(\mathcal{P}(G),x)$. 
From Lemma \ref{PG:pnqm:joinform}, we know that
\[
\mathcal{P}(G) \cong (K_1+\Gamma)[K_1,\underbrace{K_{(p-1)},\dots,K_{(p-1)}}_{\alpha},\underbrace{K_{(p-1)(q-1)},\dots,K_{(p-1)(q-1)}}_{\alpha \beta},\underbrace{K_{(q-1)},\dots,K_{(q-1)}}_{\beta}],
\]
where $\alpha = \frac{p^n-1}{p-1}$ and $\beta = \frac{q^m-1}{q-1}$. By Theorem 7 in \cite{schwenk74}, we have
\begin{equation*}
    \begin{aligned}
    \phi_A(\mathcal{P}(G),x) &=\phi(T_2,x) (x+1)^{(p-2)\alpha} (x+1)^{((p-1)(q-1)-1)\alpha \beta} (x+1)^{(q-2)\beta}\\
    &=\phi(T_2,x) (x+1)^{p^nq^m-(\alpha +1)(\beta+1)}.\\
    \end{aligned}
\end{equation*}
By using Lemma \ref{PG:pnqm:T2charpoly}, the proof is complete.
\end{proof}

\subsection{Distance spectrum of $\mathcal{P}(\El(p^n) \times \El(q^m))$}

Suppose $T := [t_{ij}]$ is the quotient matrix for a graph $\Gamma$ corresponding to the equitable partition $V(\Gamma)=V_1 \cup \dotsc \cup V_m$. Let $T^D:= [t^D_{ij}]$ be the corresponding distance quotient matrix. From (\ref{eqn:T}), it is easy to see that for each $i, j \in [m]$
\begin{equation*}\label{eqn:TD1}
    t^D_{ij}:= \begin{cases}
2|V_i|-2-t_{ii}&~\mbox{if}~i= j,\\
2|V_j|-t_{ij}&~\mbox{otherwise.}  
\end{cases}
\end{equation*} 
In the following lemma, we compute the distance quotient matrices of  $\mathcal{P}(\El(p^n) \times \El(q^m))$ corresponding to the equitable partitions given in Lemma \ref{PG:pnqm:partition1} and Lemma \ref{PG:pnqm:partition2}.
\begin{lemma}
The distance quotient matrices $T^D_1$ and $T^D_2$ of $\mathcal{P}( \El(p^n) \times \El(q^m))$ corresponding to the equitable partition given in Lemma \ref{PG:pnqm:partition1} and Lemma \ref{PG:pnqm:partition2} are
\begin{equation*}
\left[\begin{array}{cccc}
               0& p^n-1 &  (p^n-1)(q^m-1) &  q^m-1\\
     1& 2p^n-p-2 &  (2p^n-p-1)(q^m-1) &  2(q^m-1)\\
     1& 2p^n-p-1 &  2(p^n-1)(q^m-1)-(p-1)(q-1)-1 &  2q^m-q-1\\
     1& 2(p^n-1) &  (p^n-1)(2q^m-q-1) &  2q^m-q-2\\
\end{array}
\right]
\end{equation*}
and 
\begin{equation}\label{PG:pnqm:T2D}
\begin{small}
\left[ \begin{array}{cccc}
      0& (p-1)\1'_\alpha & (p-1)(q-1)\1'_{\gamma} & (q-1) \1'_\beta \\
     
     \1_\alpha& p(2J_\alpha-I_\alpha)-2J_\alpha & (p-1)(q-1)(2J_{\alpha}-I_\alpha) \otimes \1'_\beta & 2(q-1)J_{\alpha,\beta}\\
     
     \1_{\gamma} & (p-1)(2J_{\alpha}-I_\alpha) \otimes \1_\beta& (p-1)(q-1)(2J_\gamma-I_{\gamma})-I_{\gamma}& (q-1)\1_\alpha \otimes (2J_{\beta}-I_\beta)  \\
     
     \1_\beta & 2(p-1)J_{\beta, \alpha} &(p-1)(q-1)  \1'_\alpha \otimes (2J_{\beta}-I_\beta) & q(2J_{\beta}- I_\beta)-2J_{\beta}
\end{array}
\right],
\end{small}    
\end{equation}
respectively, where $\alpha = \frac{p^n-1}{p-1}$, and $\beta = \frac{q^m-1}{q-1}$ and $\gamma = \alpha \beta$.
\end{lemma}

Next, we compute the characteristic polynomial of $T^D_2$.
\begin{lemma}\label{PG:pnqm:T2Dcharpoly}
Suppose $T^D_2$ is the matrix given in  \eqref{PG:pnqm:T2D}, then 
\[\phi(T^D_2,x) = \phi(T_1^D,x) (x+(p-1)(q-1)+1)^{(\alpha-1)(\beta-1)} f(x)^{\alpha-1}g(x)^{\beta-1},\]
where $\alpha = \frac{p^n-1}{p-1}$, $\beta = \frac{q^m-1}{q-1}$, and $f(x)$ and $g(x)$ are the characteristic polynomials of the matrices  $\left[\begin{array}{cc}
     -p&  (1-p)(q^m-1)\\
     1-p& (1-p)(q-1)-1
\end{array}\right]$ and $\left[\begin{array}{cc}
     (p-1)(1-q)-1&1-q  \\
     (p^n-1)(1-q)& -q
\end{array}\right]$, respectively.
\end{lemma}
\begin{proof}
As in Lemma  \ref{PG:pnqm:T2charpoly},  it is easy to see that $\phi(T^D_1,x)$ divides $\phi(T^D_2,x)$. Let $S_1$ be the set of eigenvectors of $T^D_2$ formed from the linear independent eigenvectors of $T^D_1$. For $i \in \{1,\dotsc,\alpha-1\}$ and $j \in \{1,\dotsc,\beta-1\}$, let $v^i\in \rr^{\alpha}$ and $w^j\in \rr^{\beta}$ be the vectors defined in \eqref{PG:pnqm:T2evec}. Now, for each $i \in \{1,\dotsc,\alpha-1\}$ and $j \in \{1,\dotsc,\beta-1\}$
\begin{equation*}
    \begin{aligned}
    T_2^D \left[\begin{array}{c}
    0  \\
    \0_\alpha\\
    v^i\otimes w^j\\
\0_\beta
\end{array}
\right]
&=\left[ \begin{array}{c}
     (p-1)(q-1)\1'_{\gamma} ( v^i\otimes w^j)\\
     ((p-1)(q-1)(2J_{\alpha}-I_\alpha) \otimes \1'_\beta)( v^i\otimes w^j)\\
     ((p-1)(q-1)(2J_\gamma-I_{\gamma})-I_{\gamma})( v^i\otimes w^j)\\
     ((p-1)(q-1)  \1'_\alpha \otimes (2J_{\beta}-I_\beta)) ( v^i\otimes w^j)
\end{array}
\right]\\ &=-((p-1)(q-1)+1)\left[ \begin{array}{c}
    0  \\
    \0_\alpha\\
    v^i\otimes w^j\\
    \0_\beta
\end{array}
\right].
    \end{aligned}
\end{equation*}
Since $S_2= \{(0 ,
    \0'_\alpha,
    v^i\otimes w^j,
    \0_\beta')': 1 \leq i \leq\alpha-1, 1 \leq j \leq \beta-1 \}$ is a linearly independent set, $-((p-1)(q-1)+1)$ is an eigenvalue of $T_2^D$ with multiplicity $(\alpha-1)(\beta-1)$. Clearly, $S_1 \cup S_2$ is a set of linearly independent eigenvectors of $T_2$.  Now consider the following matrices
\[B=\left[\begin{array}{cc}
     -p&  (1-p)(q^m-1)\\
     1-p& (1-p)(q-1)-1
\end{array}\right]~\mbox{and}~ C=\left[\begin{array}{cc}
     (p-1)(1-q)-1&1-q  \\
     (p^n-1)(1-q)& -q
\end{array}\right].\]
Suppose $f(x)$ and $g(x)$ are the characteristic polynomials of the matrices $B$ and $C$, respectively. As in Lemma \ref{PG:pnqm:T2charpoly}, it is easy to prove that $f(x)$ and $g(x)$ divides $\phi(T^D_2,x)$ such that 
\[\phi(T_2^D,x) = \phi(T_1^D,x) (x+(p-1)(q-1)+1)^{(\alpha-1)(\beta-1)} f(x)^{\alpha-1}g(x)^{\beta-1}.\]
This completes the proof.
\end{proof}
In the following theorem, we compute the characteristic polynomial of the distance matrix of $\mathcal{P}(\El(p^n) \times \El(q^m))$.
\begin{theorem}\label{PG:pnqm:DM:charpoly}
Suppose $G=\El(p^n) \times \El(q^m)$. If $D$ is the distance matrix of $\mathcal{P}(G)$, then
\begin{small}
    \begin{equation*}
        \phi_D(\mathcal{P}(G),x)= \phi(T_1^D,x) (x+1)^{p^nq^m-(\alpha +1)(\beta+1)} (x+(p-1)(q-1)+1)^{(\alpha-1)(\beta-1)} f(x)^{\alpha-1}g(x)^{\beta-1},
    \end{equation*}
\end{small}where $\alpha = \frac{p^n-1}{p-1}$, $\beta = \frac{q^m-1}{q-1}$, and $f(x)$ and $g(x)$ are the characteristic polynomials of the matrices   $\left[\begin{array}{cc}
     -p&  (1-p)(q^m-1)\\
     1-p& (1-p)(q-1)-1
\end{array}\right]$ and $\left[\begin{array}{cc}
     (p-1)(1-q)-1&1-q  \\
     (p^n-1)(1-q)& -q
\end{array}\right]$, respectively.
\end{theorem}
\begin{proof}
Since the adjacency  and distance matrix of a complete graph is the same, using Theorem \ref{thm:join} along with Theorem \ref{PG:pnqm:AM:charpoly}, we get the desired result.
\end{proof}

We now discuss the adjacency and distance spectrum of the enhanced power graph of $\El(p^n) \times \El(q^m)$. We first need the following lemmas.

\begin{lemma}\label{EPG:pnqm:partition1}
The partition given in Lemma \ref{PG:pnqm:partition1} is an equitable partition for the enhanced power graph of $\El(p^n) \times \El(q^m)$.
\end{lemma}

\begin{proof}
We recall that $u$ is adjacent to $v$ in the enhanced power graph of a group $G$ (see definition \ref{defn: enhcdpowr graph}) if $u,v\in \langle w \rangle$ for some $w\in G$. Let $G=\El(p^n) \times \El(q^m)$. Observe that 
$(a_1,b_1)$ and $(a_2,b_2)$ are adjacent in $\mathcal{G}_E(G)$, if and only if $\langle a_1,a_2 \rangle$ and  $\langle b_1,b_2 \rangle$ are cyclic in $\El(p^n)$ and $\El(q^m)$, respectively. Thus,
\[|N(u)\cap V_4|=q^m-1~\mbox{for}~u \in V_2\]
and
\[|N(u) \cap V_2| = p^n-1~\mbox{for}~u \in V_4.\]
Moreover, except the edges between the elements of $V_2$ and $V_4$, there is no edge (other than the edges of $\mathcal{P}(G)$) in $\mathcal{G}_E(G)$.
Now, by using Lemma \ref{PG:pnqm:partition1}, the proof is complete.
\end{proof}

\begin{lemma}\label{EPG:pnqm:partition2}
The partition given in Lemma \ref{PG:pnqm:partition2} is an equitable partition for the enhanced power graph of $\El(p^n) \times \El(q^m)$.
\end{lemma}
\begin{proof}
The proof follows from Lemma \ref{PG:pnqm:partition2} and Lemma \ref{EPG:pnqm:partition1}.
\end{proof}
The next result follows directly from Lemma \ref{EPG:pnqm:partition2}.
\begin{lemma}
Let $G=\El(p^n) \times \El(q^m)$. Suppose $\alpha = \frac{p^n-1}{p-1}$ and $\beta = \frac{q^m-1}{q-1}$. Then,
\begin{small}
    \begin{equation*}
\mathcal{G}_E(G) \cong (K_1+\Gamma')[K_1,\underbrace{K_{(p-1)},\dots,K_{(p-1)}}_{\alpha},\underbrace{K_{(p-1)(q-1)},\dots,K_{(p-1)(q-1)}}_{\alpha \beta},\underbrace{K_{(q-1)},\dots,K_{(q-1)}}_{\beta}],        
    \end{equation*}
\end{small}where $\alpha = \frac{p^n-1}{p-1}$, $\beta = \frac{q^m-1}{q-1}$ and $\Gamma'$ is the graph with vertex set $V(\Gamma)$ and edge set $E(\Gamma) \cup \{\{i,\alpha+\alpha \beta+j\}: 1 \leq i \leq \alpha, 1 \leq j \leq \beta\}$, where $\Gamma$ is given in Figure \ref{fig}.
\end{lemma}

\subsection{Adjacency spectrum of $\mathcal{G}_E(\El(p^n) \times \El(q^m))$}\label{sec:EPG:pnqm:AM}
From Lemma \ref{EPG:pnqm:partition1}, and  Lemma \ref{EPG:pnqm:partition2}, the following result is immediate. 
\begin{lemma}\label{EPG:pnqm:T1T2}
    The quotient matrices $T_1$ and $T_2$ of $\mathcal{G}_E( \El(p^n) \times \El(q^m))$ corresponding to the equitable partitions given in Lemma \ref{PG:pnqm:partition1} and Lemma \ref{PG:pnqm:partition2}, respectively are
\[T_1= \left[\begin{array}{cccc}
     0& p^n-1 &  (p^n-1)(q^m-1) &  q^m-1\\
     1& p-2 &  (p-1)(q^m-1) &  q^m-1\\
     1& p-1 &  (p-1)(q-1)-1 &  q-1\\
     1& p^n-1 &  (p^n-1)(q-1) &  q-2\\
\end{array}
\right]\]
and 
\begin{equation}\label{EPG:pnqm:T2}
    T_2=\left[ \begin{array}{cccc}
     0& (p-1)\1'_\alpha & (p-1)(q-1)\1'_{\alpha \beta} & (q-1) \1'_\beta \\
     \1_\alpha& (p-2)I_\alpha & (p-1)(q-1)I_\alpha \otimes \1'_\beta & (q-1)\1_\alpha \1'_\beta\\
     \1_{\alpha \beta} & (p-1) I_\alpha \otimes \1_\beta & ((p-1)(q-1)-1)I_{\alpha \beta} & (q-1)\1_\alpha \otimes I_\beta  \\
     \1_\beta & (p-1)\1_\beta \1'_\alpha & (p-1)(q-1) \1'_\alpha \otimes I_\beta & (q-2) I_\beta
\end{array}
\right],
\end{equation}
where $\alpha = \frac{p^n-1}{p-1}$, and $\beta = \frac{q^m-1}{q-1}$.
\end{lemma}

\begin{lemma}\label{EPG:pnqm:T2charpoly}
Suppose $T_2$ is the matrix given in \eqref{EPG:pnqm:T2}. Then
\[\phi(T_2,x) = \phi(T_1,x) (x-pq+p+q)^{(\alpha-1)(\beta-1)} f(x)^{\alpha-1}g(x)^{\beta-1},\]
where $\alpha = \frac{p^n-1}{p-1}$, $\beta = \frac{q^m-1}{q-1}$, and $f(x)$ and $g(x)$ are the characteristic polynomials of the matrices $\left[\begin{array}{cc}
     p-2&  (p-1)(q^m-1)\\
     p-1& (p-1)(q-1)-1
\end{array}\right]$ and $\left[\begin{array}{cc}
     (p-1)(q-1)-1&q-1  \\
     (p^n-1)(q-1)& q-2
\end{array}\right]$, respectively.
\end{lemma}
\begin{proof}
At this point, it is evidently clear that $\phi(T_1,x)$ divides $\phi(T_2,x)$. Let $S_1$ be the set of eigenvectors of $T_2$ formed from the linear independent eigenvectors of $T_1$. For $i \in \{1,\dotsc,\alpha-1\}$ and $j \in \{1,\dotsc,\beta-1\}$, let $v^i\in \rr^{\alpha}$ and $w^j\in \rr^{\beta}$ be the vectors defined in \eqref{PG:pnqm:T2evec}. From Lemma \ref{PG:pnqm:T2charpoly}, it is obvious that
\begin{equation*}
    \begin{aligned}
    T_2 \left[\begin{array}{c}
    0  \\
    \0_\alpha\\
    v^i\otimes w^j\\
\0_\beta
\end{array}
\right]
 =(pq-p-q)\left[ \begin{array}{c}
    0  \\
    \0_\alpha\\
    v^i\otimes w^j\\
    \0_\beta
\end{array}
\right],
    \end{aligned}
\end{equation*}
for each $i \in \{1,\dotsc,\alpha-1\}$ and $j \in \{1,\dotsc,\beta-1\}$ and hence $pq-p-q$ is an eigenvalue of $T_2$ with multiplicity $(\alpha-1)(\beta-1)$. Also, these $(\alpha-1)(\beta-1)$ eigenvectors and the vectors in $S_1$ are linearly independent. Suppose $f(x)$ and $g(x)$ are the characteristic polynomials of the matrices
\[B=\left[\begin{array}{cc}
     p-2&  (p-1)(q^m-1)\\
     p-1& (p-1)(q-1)-1
\end{array}\right]~\mbox{and}~ C=\left[\begin{array}{cc}
     (p-1)(q-1)-1&q-1  \\
     (p^n-1)(q-1)& q-2
\end{array}\right],\]
respectively. Similar to the proof of Lemma \ref{PG:pnqm:T2charpoly}, it is easy to see that $f(x)$ and $g(x)$ divides $\phi(T_2,x)$ such that
\[\phi(T_2,x) = \phi(T_1,x) (x-pq+p+q)^{(\alpha-1)(\beta-1)} f(x)^{\alpha-1}g(x)^{\beta-1}.\]
The proof is complete.
\end{proof}

Now, we are ready to state the result analogous to Theorem \ref{PG:pnqm:AM:charpoly} for the characteristic polynomial of the adjacency matrix of the $\mathcal{G}_E(\El(p^n) \times \El(q^m))$. The proof is omitted since it is similar to the proof of Theorem \ref{PG:pnqm:AM:charpoly}.
\begin{theorem}
Let $G=\El(p^n) \times \El(q^m)$. If $A$ is the adjacency matrix of $\mathcal{G}_E(G)$, then
\[\phi_A(\mathcal{G}_E(G),x)= \phi(T_1,x) (x+1)^{(p^nq^m-(\alpha +1)(\beta+1))} (x-pq+p+q)^{(\alpha-1)(\beta-1)} f(x)^{\alpha-1}g(x)^{\beta-1},\]
where $\alpha = \frac{p^n-1}{p-1}$, $\beta = \frac{q^m-1}{q-1}$, and $f(x)$ and $g(x)$ are the characteristic polynomials of the matrices  $\left[\begin{array}{cc}
     p-2&  (p-1)(q^m-1)\\
     p-1& (p-1)(q-1)-1
\end{array}\right]$ and $\left[\begin{array}{cc}
     (p-1)(q-1)-1&q-1  \\
     (p^n-1)(q-1)& q-2
\end{array}\right]$, respectively.
\end{theorem}

\subsection{Distance spectrum of $\mathcal{G}_E(\El(p^n) \times \El(q^m))$}
Using \eqref{eqn:TD1} and Lemma \ref{EPG:pnqm:T1T2}, we can compute the distance quotient matrices of $\mathcal{G}_E( \El(p^n) \times \El(q^m))$ corresponding to the equitable partition given in Lemma \ref{PG:pnqm:partition1} and Lemma \ref{PG:pnqm:partition2}. These matrices are described in the following lemma.
\begin{lemma}\label{EPG:pnqm:T1DT2D}
The distance quotient matrices $T^D_1$ and $T^D_2$ of $\mathcal{G}_E( \El(p^n) \times \El(q^m))$ corresponding to the equitable partition given in Lemma \ref{PG:pnqm:partition1} and Lemma \ref{PG:pnqm:partition2} are

\[\left[\begin{array}{cccc}
               0& p^n-1 &  (p^n-1)(q^m-1) &  q^m-1\\
     1& 2p^n-p-2 &  (2p^n-p-1)(q^m-1) &  q^m-1\\
     1& 2p^n-p-1 &  (p-1)(q-1)(2\gamma-1)-1 &  2q^m-q-1\\
     1& p^n-1 &  (p^n-1)(2q^m-q-1) &  2q^m-q-2\\
\end{array}
\right]\]
and
\begin{equation}\label{EPG:pnqm:T2D}
    \begin{small}
  \left[ \begin{array}{cccc}
      0& (p-1)\1'_\alpha & (p-1)(q-1)\1'_{\gamma} & (q-1) \1'_\beta \\
     
     \1_\alpha& p(2J_\alpha-I_\alpha)-2J_\alpha & (p-1)(q-1)(2J_{\alpha}-I_\alpha) \otimes \1'_\beta & (q-1)J_{\alpha,\beta}\\
     
     \1_{\gamma} & (p-1)(2J_{\alpha}-I_\alpha) \otimes \1_\beta & (p-1)(q-1)(2J_\gamma-I_{\gamma})-I_{\gamma}& (q-1)\1_\alpha \otimes (2J_{\beta}-I_\beta)  \\
     
     \1_\beta & (p-1)J_{\beta, \alpha} &(p-1)(q-1)  \1'_\alpha \otimes (2J_{\beta}-I_\beta) & q(2J_{\beta}- I_\beta)-2J_{\beta}
\end{array}
\right],      
    \end{small}
\end{equation}
respectively, where $\alpha = \frac{p^n-1}{p-1}$, and $\beta = \frac{q^m-1}{q-1}$ and $\gamma = \alpha \beta$.
\end{lemma}

In the following lemma, we compute the characteristic polynomial of $T_2^D$. 
\begin{lemma}\label{EPG:pnqm:T2Dcharpoly}
Suppose $T^D_2$ is the matrix given in  \eqref{EPG:pnqm:T2D}, then 
\[\phi(T^D_2,x) = \phi(T_1^D,x) (x+(p-1)(q-1)+1)^{(\alpha-1)(\beta-1)} f(x)^{\alpha-1}g(x)^{\beta-1},\]
where $\alpha = \frac{p^n-1}{p-1}$, $\beta = \frac{q^m-1}{q-1}$, and $f(x)$ and $g(x)$ are the characteristic polynomials of the matrices  $\left[\begin{array}{cc}
     -p&  (1-p)(q^m-1)\\
     1-p& (1-p)(q-1)-1
\end{array}\right]$ and $\left[\begin{array}{cc}
     (p-1)(1-q)-1&1-q  \\
     (p^n-1)(1-q)& -q
\end{array}\right]$, respectively.
\end{lemma}
\begin{proof}
It is easy to see that $\phi(T_1^D,x)$ divides $\phi(T_2^D,x)$. Also, it follows directly from Lemma \ref{PG:pnqm:T2Dcharpoly}, that for each $i \in \{1,\dotsc,\alpha-1\}$ and $j \in \{1,\dotsc,\beta-1\}$
\begin{equation*}
    \begin{aligned}
    T_2^D \left[\begin{array}{c}
    0  \\
    \0_\alpha\\
    v^i\otimes w^j\\
\0_\beta
\end{array}
\right]
=-((p-1)(q-1)+1)\left[ \begin{array}{c}
    0  \\
    \0_\alpha\\
    v^i\otimes w^j\\
    \0_\beta
\end{array}
\right],
    \end{aligned}
\end{equation*}
where $v^i\in \rr^{\alpha}$ and $w^j\in \rr^{\beta}$ are defined in \eqref{PG:pnqm:T2evec}. 
Thus, we can conclude that $-(1+(p-1)(q-1))$ is an eigenvalue of $T_2^D$ with multiplicity $(\alpha-1)(\beta-1)$.  Now, consider the following matrices
\[B=\left[\begin{array}{cc}
     -p&  (1-p)(q^m-1)\\
     1-p& (1-p)(q-1)-1
\end{array}\right]~\mbox{and}~ C=\left[\begin{array}{cc}
     (p-1)(1-q)-1&1-q  \\
     (p^n-1)(1-q)& -q
\end{array}\right].\]
Similar to the proof of Lemma \ref{PG:pnqm:T2Dcharpoly}, we can conclude that every eigenvalue of $B$ is an eigenvalue of $T_2^D$ with multiplicity $\alpha-1$, and every eigenvalue of $C$ is an eigenvalue of $T_2^D$ with multiplicity $\beta-1$. Also, if $f(x)$ and $g(x)$ are the characteristic polynomials of the matrices $B$ and $C$, respectively, then
\[\phi(T_2^D,x) = \phi(T_1^D,x) (x+(p-1)(q-1)+1)^{(\alpha-1)(\beta-1)} f(x)^{\alpha-1}g(x)^{\beta-1}.\]
This completes the proof.
\end{proof}

The next theorem is analogous to Theorem \ref{PG:pnqm:DM:charpoly} and computes the characteristic polynomial of the distance matrix of $\mathcal{G}_E( \El(p^n) \times \El(q^m))$. The proof is  similar to the proof of Theorem \ref{PG:pnqm:DM:charpoly}.

\begin{theorem}
Suppose $G=\El(p^n) \times \El(q^m)$. If $D$ is the distance matrix of
$\mathcal{G}_E(G)$, then the characteristic polynomial $\phi_D(\mathcal{G}_E(G),x)$ of $D$ is
\[\phi(T_1^D,x) (x+1)^{(p^nq^m-(\alpha +1)(\beta+1))} (x+(p-1)(q-1)+1)^{(\alpha-1)(\beta-1)} f(x)^{\alpha-1}g(x)^{\beta-1},\]
where $\alpha = \frac{p^n-1}{p-1}$, $\beta = \frac{q^m-1}{q-1}$, and $f(x)$ and $g(x)$ are the characteristic polynomials of the matrices  $\left[\begin{array}{cc}
     -p&  (1-p)(q^m-1)\\
     1-p& (1-p)(q-1)-1
\end{array}\right]$ and $\left[\begin{array}{cc}
     (p-1)(1-q)-1&1-q  \\
     (p^n-1)(1-q)& -q
\end{array}\right]$, respectively.
\end{theorem}

\section{The finite abelian group $\El(p^n) \times \mathbb{Z}_m$} 
\label{sec:finiteabelgrp} 

We conclude the  paper by  determining the distance spectra of  the enhanced power graph of finite abelian group $\El(p^n) \times \mathbb{Z}_m$, where $n \geq 2$ and $\text{gcd}(p,m)=1$.
\begin{theorem}
\label{thm:elabelgen} 
Let $G=\El(p^n)\times \mathbb{Z}_m$, where $m$ is a positive integer such that $\text{gcd}(m,p)=1$. Suppose $\alpha=\frac{p^n-1}{p-1}$, then
$\phi_D(\mathcal{G}_E(G),x)$ is 
\begin{small}
    \begin{equation*}
        (x+mp-m+1)^{\alpha-1} (x+1)^{(mp-m-1)\alpha+m-1}(x^2+(mp+2-2mp^n)x+m^2(p^n-p)-2mp^n+mp+1).
    \end{equation*}
\end{small}
\end{theorem}
\begin{proof}
By \cite[Corollary 4.2]{bera-dey-JGT-2022}, the dominating vertices of the graph
$\mathcal{G}_E(G)$ are precisely the vertices of the form $(e, a) $, where $e$ is the identity of $\El(p^n)$ and $a \in \mathbb{Z}_m.$
From Theorem 3.2 of \cite{bera-dey-sajal}, we know that the number of connected components of $\mathcal{G}_E(\El(p^n)\setminus \{e\})$ is $\frac{p^n-1}{p-1}$. Let $C_1,  \dots, C_{\alpha}$ be the components of $\mathcal{G}_E(\El(p^n)\setminus \{e\})$. Since $\text{gcd}(p,m)=1$, the sets
\[V_1:= \{(e,a): a \in \mathbb{Z}_m \},~\mbox{and}~V_{i}:= \{ ( b,a): b \in C_{i-1},  a \in \mathbb{Z}_m\},~\mbox{for}~ 2 \leq i \leq \alpha+1, \]
form an equitable partition for $\mathcal{G}_E(G)$. It is clear that for each $2 \leq i \leq \alpha+1$, $V_i$ is a complete graph on $m(p-1)$ vertices. 
Thus,
\begin{equation}\label{EPG:EpnZm:joinform}
  \mathcal{G}_E(G) \cong K_{1,\alpha}[K_m, \underbrace{K_{m(p-1)},\dotsc,K_{m(p-1)}}_{\alpha}].  
\end{equation}
 Suppose $T^D$ is the distance quotient matrix of $\mathcal{G}_E(G)$ corresponding to the equitable partition given in \eqref{EPG:EpnZm:joinform}. Then 
\[T^D= \left[\begin{array}{cccc}
     m-1 & m(p-1)\1_\alpha' \\
     m\1_\alpha & 2m(p-1)J_{\alpha}-(mp-m+1)I_{\alpha}
\end{array}\right].\]
By Theorem \ref{thm:join}, we have
\begin{equation*}
    \begin{aligned}
    \phi_D(\mathcal{G}_E(G), x) 
    &=\phi(T^D,x)  (x+1)^{(mp-m-1)\alpha+(m-1)}.
    \end{aligned}
\end{equation*}
We now compute the characteristic polynomial of $T^D$. Since $p$ is a prime and $n\geq 2$, $2mp^n-mp-m-1$ is not an eigenvalue of $T^D$. Thus,
\begin{equation*}
    \begin{aligned}
     \phi(T^D,x) &=\det(xI_{\alpha+1}-T^D)\\ &= \left|\begin{array}{cccc}
     x-m+1& -(p-1)m\1_{\alpha}' \\
     -m\1_{\alpha}& (x+mp-m+1)I_{\alpha}-2m(p-1)J_{\alpha}
\end{array}\right|\\
&= \frac{m}{x-(2mp^n-mp-m-1)}\det (T'),
\end{aligned}
\end{equation*}
where
\[T'=\left[ \begin{array}{cccc}
   \displaystyle   \frac{(x-m+1)(x-2mp^n+mp+m+1)}{m}& -m(p-1)\1_{\alpha}' \\
     -(x-2mp^n+mp+m+1)\1_{\alpha}& (x+mp-m+1)I_{\alpha}-2m(p-1)J_{\alpha}
\end{array}\right].
\]
Replacing the first column of $T'$ by the sum of all the columns, the determinant of $T'$ is equal to
\begin{small}
\begin{equation*}
    \begin{aligned}
 \left|\begin{array}{cccc}
\displaystyle     \frac{(x-m+1)(x-2mp^n+mp+m+1)}{m}-m(p^n-1)\hspace{-3 mm}& -m(p-1)\1_{\alpha}' \\
     \0& \hspace{-3 mm}  (x+mp-m-1)I_{\alpha}-2m(p-1)J_{\alpha}
\end{array}\right|. 
\end{aligned}
\end{equation*}
\end{small}
Hence,
\begin{equation*}
    \phi(T^D,x)=
(x^2+(mp+2-2mp^n)x+m^2(p^n-p)-2mp^n+mp+1)(x+mp-m+1)^{\alpha-1}.
\end{equation*}
The proof is complete.
\end{proof}

Setting $m=1$ in Theorem \ref{thm:elabelgen}, we immediately have the distance spectra of the enhanced power graph of elementary abelian groups. This is stated below.

\begin{theorem}
\label{thm:elmntAbelGrp}
Suppose $\alpha=\frac{p^n-1}{p-1}$. Then
\[\phi_D(\mathcal{G}_E(\El(p^n)),x)= (x+p)^{\alpha-1} (x+1)^{(p-2)\alpha}(x^2-(2p^n-p-2)x-(p^n-1)).\]
In particular, the eigenvalues of $\mathcal{P}(\El(p^n))$ are $-1$ with multiplicity $(n-2)\alpha$, $-p$ with multiplicity $\alpha-1$ and two simple eigenvalues $\lambda_{1,2} = \frac{(2p^n-p-2)\pm \sqrt{(2p^n-p-2)^2+4(p^n-1)}}{2}.$
\end{theorem}

\subsection*{Acknowledgement}
 The first author thanks the Prime Minister's Research Fund for the funding. The second and third authors acknowledge SERB-National Post Doctoral Fellowship (File No. PDF/\allowbreak2021/001899) and (File No. PDF/\allowbreak2022/000269), respectively for the preparation of this work and profusely thank the Science and Engineering Research Board, Govt. of India. The authors also acknowledge ideal working conditions in the Department of Mathematics, Indian Institute of Science, Bangalore.

\bibliography{refs}

\begin{thebibliography}{10}
\expandafter\ifx\csname url\endcsname\relax
  \def\url#1{\texttt{#1}}\fi
\expandafter\ifx\csname urlprefix\endcsname\relax\def\urlprefix{URL }\fi
\expandafter\ifx\csname href\endcsname\relax
  \def\href#1#2{#2} \def\path#1{#1}\fi

\bibitem{directedgrphcompropofsemgrpkq3}
A.~V. Kelarev, S.~J. Quin, Directed graph and combinatorial properties of
  semigroups, Journal of Algebra 251 (2002) 16--26.

\bibitem{combinatorialpropertyandpowergraphsofgroupskq1}
A.~V. Kelarev, S.~J. Quin, A combinatorial property and power graphs of groups,
  Contributions to General Algebra 12 (2000) 229--235.

\bibitem{AalipourcameronEJC}
G.~Aalipour, S.~Akbari, P.~J. Cameron, R.~Nikandish, F.~Shaveisi, On the
  structure of the power graph and the enhanced power graph of a group, The
  Electronic Journal of Combinatorics 24(3) (2017) P3.16.

\bibitem{Neumann1976APO}
B.~H. Neumann, A problem of paul erd{\"o}s on groups, Journal of the Australian
  Mathematical Society 21 (1976) 467 -- 472.

\bibitem{David}
D.~Peifer, An introduction to combinatorial group theory and the word problem,
  Mathematics Magazine 70~(1) (1997) 3--10.

\bibitem{Zahirovienhnacedpwrgraph}
S.~Zahirovi\'{c}, I.~Bo\~v{s}njak, R.~Madar\'{a}sz, A study of enhanced power
  graphs of finite groups, Journal of Algebra and its Applications 19~(4)
  (2020) 2050062.

\bibitem{enhancedpwrgrapbb3}
S.~Bera, A.~K. Bhuniya, On enhanced power graphs of finite groups, Journal of
  Algebra and Its Applications. 17~(8) (2018) 1850146.

\bibitem{ma-She}
X.~Ma, Y.~She, The metric dimension of the enhanced power graph of a finite
  group, Journal of Algebra and Its Applications 19~(01) (2020) 2050020.

\bibitem{chatto-pani-LAMA2015}
S.~Chattopadhyay, P.~Panigrahi, On laplacian spectrum of power graphs of finite
  cyclic and dihedral groups, Linear and Multilinear Algebra 63 (2015)
  190--198.

\bibitem{mehranian-gholami-ashrafi-LAMA2016}
Z.~Mehranian, A.~Gholami, A.~R. Ashrafi, The {S}pectra of power graphs of
  certain finite groups, Linear and Multilinear Algebra 65 (2017) 1003--1010.

\bibitem{hamzehashrafiFILOMAT2017}
A.~Hamzeh, A.~R. Ashrafi, Spectrum and {L}{-}spectrum of the power graph and
  its main supergraph for certain finite groups, Filomat 16 (2017) 5323--5334.

\bibitem{wani2020}
M.~Wani, S.~K, On distance spectra of power graphs of finite groups,
  International Journal of Mathematics Trends and Technology 66 (2020)
  150--160.

\bibitem{Andries2010}
W.~H.~H. Andries E.~Brouwer, Spectra of graphs, Springer, 2010.

\bibitem{stevanovic-AMC-2002}
D.~Stevanovic, Large sets of long distance equienergetic graphs, Ars
  Mathematica Contemporanea 2 (2009) 35--40.

\bibitem{CHATTOPADHYAY2018730}
S.~Chattopadhyay, P.~Panigrahi, F.~Atik, Spectral radius of power graphs on
  certain finite groups, Indagationes Mathematicae 29~(2) (2018) 730--737.

\bibitem{Mirzargar}
M.~Mirzargar, A.~R. Ashrafi, M.~J. Nadjafi-Arani, On the power graph of a
  finite group, Filomat 26~(6) (2012) 1201--1208.

\bibitem{mehranian-gholami-ashrafi-IJGT}
Z.~Mehranian, A.~Gholami, A.~R. Ashrafi, A note on the power graph of a finite
  group, International Journal of Group Theory 5 (2016) 1--10.

\bibitem{schwenk74}
A.~J. Schwenk, Computing the characteristic polynomial of a graph, in: Graphs
  and Combinatorics, Vol. 406 of Lecture Notes in Mathematics, Springer,
  Berlin, Heidelberg, 1974, pp. 153--172.

\bibitem{bera-dey-JGT-2022}
S.~Bera, H.~K. Dey, On the proper enhanced power graphs of finite nilpotent
  groups, Journal of Group Theory 25 (2022) 1109--1131.

\bibitem{bera-dey-sajal}
S.~Bera, H.~K. Dey, S.~K. Mukherjee, On the connectivity of enhanced power
  graphs of finite groups, Graphs and Combinatorics 37(2) (2021) 591--603.

\end{thebibliography}
{\em Authors' address}:\\
{\em Anita Arora}, Department of Mathematics, IISc Bangalore, Karnataka, India.
 E-mail: \texttt{anitaarora2018@iitkalumni.org}
 \\
\noindent 
 {\em Hiranya Kishore Dey}, Department of Mathematics, IISc Bangalore, Karnataka, India.
 E-mail: \texttt{hiranya.dey@gmail.com}\\
\noindent 
{\em Shivani Goel}, Department of Mathematics, IISc Bangalore, Karnataka, India.
 E-mail: \texttt{shivani.goel.maths@gmail.com} 
\end{document}